\documentclass[10pt]{amsart}
\usepackage[english]{babel}
\usepackage{amsmath,amsthm}
\usepackage{amsfonts}
\usepackage{comment}
\usepackage{enumerate}
\usepackage{graphicx}
\usepackage{hyperref}
\usepackage{subcaption}
\usepackage{color}
\usepackage[all]{xy}
\usepackage{thmtools, thm-restate}
\usepackage{amssymb,tikz}
\usepackage{quiver}
\usepackage{enumitem}
\usepackage{multicol}
\usepackage{alltt}

\usepackage[top=1.23in, bottom=1.22in, left=1.23in, right=1.23in]{geometry}

\newcommand*\circled[1]{\tikz[baseline=(char.base)]{
            \node[shape=circle,draw,inner sep=1pt] (char) {#1};}}

\newcommand{\red}{\textcolor{red}}

\newtheorem{thm}{Theorem}[section]
\newtheorem{cor}[thm]{Corollary}
\newtheorem{lem}[thm]{Lemma}
\newtheorem{prop}[thm]{Proposition}

\newtheorem*{claim}{Claim}
\newtheorem{case}{Case}

\theoremstyle{definition}
\newtheorem{defn}[thm]{Definition}
\newtheorem*{conditions}{Working Conditions}
\newtheorem{notation}[thm]{Notation}
\newtheorem{ques}{Question}

\theoremstyle{remark}
\newtheorem{rem}[thm]{Remark}

\numberwithin{equation}{section}

\newcommand{\norm}[1]{\lVert #1 \rVert^2}

\newcommand{\spin}{\ifmmode{\rm Spin}\else{${\rm spin}$\ }\fi}
\newcommand{\spinc}{\ifmmode{{\rm Spin}^c}\else{${\rm spin}^c$}\fi}

\newcommand{\Z}{\mathbb{Z}}

\newcommand{\Q}{\mathbb{Q}}

\DeclareMathOperator{\tors}{tors}

\newcommand{\Lsum}{\#_i L(p_i,q_i)}
\newcommand{\Xsum}{\natural_i X(p_i,q_i)}

\newcommand{\plumb}{\entrymodifiers={+[o][F-]} \xymatrix@C=8pt}
\newcommand{\el}{\ar@{-}[r]}
\newcommand{\ed}{\ar@{..}[r] }

\DeclareMathOperator{\supp}{Supp}
\DeclareMathOperator{\aut}{Aut}

\begin{document}

\title{Definite fillings of lens spaces}%

\author{Paolo Aceto}%
\address {Université de Lille}
\email{paoloaceto@gmail.com }

\author{Duncan McCoy}%
\address {Universit\'{e} du Qu\'{e}bec \`{a} Montr\'{e}al}
\email{mc\_coy.duncan@uqam.ca}

\author{JungHwan Park}%
\address {Korea Advanced Institute of Science and Technology}
\email{jungpark0817@kaist.ac.kr}
\date{}%

\begin{abstract}
This paper considers the problem of determining the smallest (as measured by the second Betti number) smooth negative-definite filling of a lens space. The main result is to classify those lens spaces for which the associated negative-definite canonical plumbing is minimal. The classification takes the form of a list of 10 ``forbidden'' subgraphs that cannot appear in the plumbing graph if the corresponding plumbed 4-manifold is minimal. We also show that whenever the plumbing is minimal any other negative-definite filling for the given lens space has the same intersection form up to addition of diagonal summands. Consequences regarding smooth embeddings of lens spaces in 4-manifolds are also discussed.
\end{abstract}
\maketitle

\section{Introduction}
Many problems in low-dimensional topology naturally lead to the question of determining if a given 3-manifold bounds a 4-manifold with certain constraints on its topology. In this article, we consider \emph{definite fillings}, i.e.\ smooth compact 4-manifolds with definite intersection forms. Throughout this article, we will always work in the smooth category unless otherwise specified, and all manifolds are oriented. If a compact 4-manifold $X$ is a filling of a closed 3-manifold $Y$, then we say that the intersection form $Q_X$ \emph{fills} $Y$ or that $Q_X$ is \emph{bounded} by $Y$.

In its most general form the problem that we consider is the following:
$$\text{Given a 3-manifold $Y$, classify all definite intersection forms filling $Y$.}$$ 
In this context, Donaldson's diagonalization theorem~\cite{Donaldson:1987-1} gives a complete answer when $Y\cong S^3$: the only definite intersection forms filling $S^3$ are the diagonal ones. More work in this direction includes~\cite{Froyshov:1995-1, Owens-Strle:2006-1, Owens-Strle:2012-1, Choe-Park:2018-2, Scaduto:2018-1, Golla-Scaduto:2019-1}. 

 A \emph{lens space} $L(p,q)$ is the result of $-p/q$-Dehn surgery on the unknot in $S^3$ for some coprime integers $p$ and $q$ with $p>q>0$. Each lens space $L(p,q)$ is the boundary of a \emph{canonical negative-definite plumbing} $X(p,q)$, which is a smooth negative-definite 4-manifold obtained by plumbing disk-bundles over spheres according to a \emph{linear plumbing graph} of the form 
$$\begin{tikzpicture}[xscale=1.0,yscale=1,baseline={(0,0)}]
    \node at (1-0.1, .4) {$-a_1$};
    \node at (2-0.1, .4) {$-a_2$};
        \node at (3-0.1, .4) {$-a_3$};
    \node at (5-0.1, .4) {$-a_{n}$};
    \node (A1_1) at (1, 0) {$\bullet$};
    \node (A1_2) at (2, 0) {$\bullet$};
    \node (A1_3) at (3, 0) {$\bullet$};
    \node (A1_4) at (4, 0) {$\cdots$};
    \node (A1_5) at (5, 0) {$\bullet$};
    \path (A1_2) edge [-] node [auto] {$\scriptstyle{}$} (A1_3);
    \path (A1_3) edge [-] node [auto] {$\scriptstyle{}$} (A1_4);
        \path (A1_4) edge [-] node [auto] {$\scriptstyle{}$} (A1_5);
    \path (A1_1) edge [-] node [auto] {$\scriptstyle{}$} (A1_2);
  \end{tikzpicture}$$
where the weights $a_i$ satisfy $a_i\geq 2$ and are uniquely determined by the continued fraction expansion
$$\frac{p}{q} = a_1 - \cfrac{1}{a_2 - \cfrac{1}{\ddots - \cfrac{1}{a_n}}}$$
of $p/q$.

A connected sum of lens spaces $\#_i L(p_i,q_i)$ bounds the boundary connected sum of canonical negative-definite plumbings  $\natural_i X(p_i,q_i)$ and the associated plumbing graph is given by the disjoint union of the linear plumbing graphs corresponding to the summands. 

Note that if $Y$ bounds the negative-definite form $Q_X$, then by blowing up $X$ with copies of $\overline{\mathbb{CP}}^2$ we see that $Y$ also bounds $Q_X \oplus \langle -1\rangle^n$.
The main result of this paper is a classification for which (connected sums of) lens spaces the canonical negative-definite filling is $b_2$-minimal or, equivalently, those lens spaces filled by a unique intersection form up to addition of $\langle -1\rangle$ summands.

\begin{thm}\label{thm:main}
If $L=\#_i L(p_i,q_i)$ is a connected sum of lens spaces and $\natural_i X(p_i,q_i)$ is the corresponding boundary connected sum of canonical negative-definite plumbings, then the following are equivalent:
\begin{enumerate}[label=(\roman*), font=\upshape]
\item\label{it:min_filling} every smooth negative-definite filling $X$ of $L$ satisfies
$$b_2(X) \geq b_2(\natural_i X(p_i,q_i));$$
\item\label{it:rigid} every smooth negative-definite filling $X$ of $L$ satisfies 
$$Q_X \cong Q_{\Xsum} \oplus \langle-1\rangle^n$$
for some integer $n\geq 0$;
\item\label{it:combinatorial} the canonical negative-definite linear plumbing graph associated to $L$ does not contain any of the following configurations as an induced subgraph\footnote{Recall that an induced subgraph of a graph $G$ is obtained by taking a subset of the vertices and all edges in $G$ connecting pairs of these vertices.}:
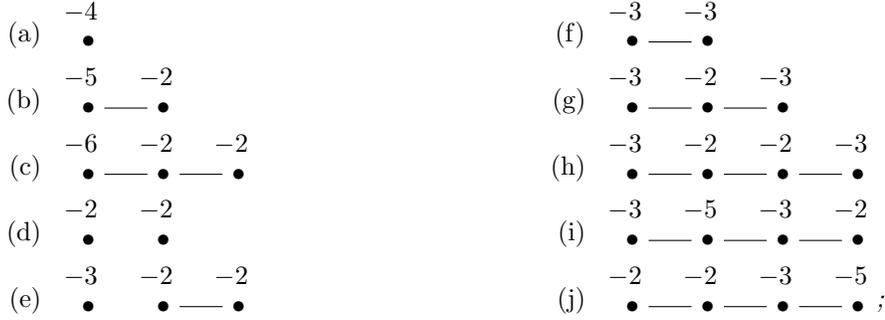
\begin{multicols}{2}
\begin{enumerate}[label=(\alph*),font=\upshape]
\item\label{it:4} $\begin{tikzpicture}[xscale=1.0,yscale=1,baseline={(0,0)}]
    \node at (1-0.1, .4) {$-4$};
    \node (A_1) at (1, 0) {$\bullet$};
  \end{tikzpicture}$
\item\label{it:52} $\begin{tikzpicture}[xscale=1.0,yscale=1,baseline={(0,0)}]
    \node at (1-0.1, .4) {$-5$};
    \node at (2-0.1, .4) {$-2$};
    \node (A_1) at (1, 0) {$\bullet$};
    \node (A_2) at (2, 0) {$\bullet$};
        \path (A_1) edge [-] node [auto] {$\scriptstyle{}$} (A_2);
  \end{tikzpicture}$
\item\label{it:622} $\begin{tikzpicture}[xscale=1.0,yscale=1,baseline={(0,0)}]
    \node at (1-0.1, .4) {$-6$};
    \node at (2-0.1, .4) {$-2$};
    \node at (3-0.1, .4) {$-2$};
    \node (A_1) at (1, 0) {$\bullet$};
    \node (A_2) at (2, 0) {$\bullet$};
    \node (A_3) at (3, 0) {$\bullet$};
        \path (A_1) edge [-] node [auto] {$\scriptstyle{}$} (A_2);
    \path (A_2) edge [-] node [auto] {$\scriptstyle{}$} (A_3);
  \end{tikzpicture}$
\item\label{it:2-2} $\begin{tikzpicture}[xscale=1.0,yscale=1,baseline={(0,0)}]
    \node at (1-0.1, .4) {$-2$};
    \node at (2-0.1, .4) {$-2$};
    \node (A_1) at (1, 0) {$\bullet$};
    \node (A_2) at (2, 0) {$\bullet$};
  \end{tikzpicture}$
\item\label{it:3-22}  $\begin{tikzpicture}[xscale=1.0,yscale=1,baseline={(0,0)}]
    \node at (1-0.1, .4) {$-3$};
    \node at (2-0.1, .4) {$-2$};
    \node at (3-0.1, .4) {$-2$};
    \node (A_1) at (1, 0) {$\bullet$};
    \node (A_2) at (2, 0) {$\bullet$};
    \node (A_3) at (3, 0) {$\bullet$};
    \path (A_2) edge [-] node [auto] {$\scriptstyle{}$} (A_3);
  \end{tikzpicture}$
\item\label{it:33}
$\begin{tikzpicture}[xscale=1.0,yscale=1,baseline={(0,0)}]
    \node at (1-0.1, .4) {$-3$};
    \node at (2-0.1, .4) {$-3$};
    \node (A1_1) at (1, 0) {$\bullet$};
    \node (A1_2) at (2, 0) {$\bullet$};
    \path (A1_1) edge [-] node [auto] {$\scriptstyle{}$} (A1_2);
  \end{tikzpicture}$
\item\label{it:323} $\begin{tikzpicture}[xscale=1.0,yscale=1,baseline={(0,0)}]
    \node at (1-0.1, .4) {$-3$};
    \node at (2-0.1, .4) {$-2$};
    \node at (3-0.1, .4) {$-3$};
    \node (A1_1) at (1, 0) {$\bullet$};
    \node (A1_2) at (2, 0) {$\bullet$};
    \node (A1_3) at (3, 0) {$\bullet$};
    \path (A1_2) edge [-] node [auto] {$\scriptstyle{}$} (A1_3);
    \path (A1_1) edge [-] node [auto] {$\scriptstyle{}$} (A1_2);
  \end{tikzpicture}$
\item\label{it:3223}
$\begin{tikzpicture}[xscale=1.0,yscale=1,baseline={(0,0)}]
    \node at (1-0.1, .4) {$-3$};
    \node at (2-0.1, .4) {$-2$};
    \node at (3-0.1, .4) {$-2$};
    \node at (4-0.1, .4) {$-3$};
    \node (A1_1) at (1, 0) {$\bullet$};
    \node (A1_2) at (2, 0) {$\bullet$};
    \node (A1_3) at (3, 0) {$\bullet$};
    \node (A1_4) at (4, 0) {$\bullet$};
    \path (A1_2) edge [-] node [auto] {$\scriptstyle{}$} (A1_3);
    \path (A1_3) edge [-] node [auto] {$\scriptstyle{}$} (A1_4);
    \path (A1_1) edge [-] node [auto] {$\scriptstyle{}$} (A1_2);
  \end{tikzpicture}$
\item\label{it:3532} 
$\begin{tikzpicture}[xscale=1.0,yscale=1,baseline={(0,0)}]
    \node at (1-0.1, .4) {$-3$};
    \node at (2-0.1, .4) {$-5$};
    \node at (3-0.1, .4) {$-3$};
    \node at (4-0.1, .4) {$-2$};
    \node (A1_1) at (1, 0) {$\bullet$};
    \node (A1_2) at (2, 0) {$\bullet$};
    \node (A1_3) at (3, 0) {$\bullet$};
    \node (A1_4) at (4, 0) {$\bullet$};
    \path (A1_2) edge [-] node [auto] {$\scriptstyle{}$} (A1_3);
    \path (A1_3) edge [-] node [auto] {$\scriptstyle{}$} (A1_4);
    \path (A1_1) edge [-] node [auto] {$\scriptstyle{}$} (A1_2);
  \end{tikzpicture}$
\item\label{it:2235} $\begin{tikzpicture}[xscale=1.0,yscale=1,baseline={(0,0)}]
    \node at (1-0.1, .4) {$-2$};
    \node at (2-0.1, .4) {$-2$};
    \node at (3-0.1, .4) {$-3$};
    \node at (4-0.1, .4) {$-5$};
    \node (A1_1) at (1, 0) {$\bullet$};
    \node (A1_2) at (2, 0) {$\bullet$};
    \node (A1_3) at (3, 0) {$\bullet$};
    \node (A1_4) at (4, 0) {$\bullet$};
    \path (A1_2) edge [-] node [auto] {$\scriptstyle{}$} (A1_3);
    \path (A1_3) edge [-] node [auto] {$\scriptstyle{}$} (A1_4);
    \path (A1_1) edge [-] node [auto] {$\scriptstyle{}$} (A1_2);
  \end{tikzpicture}$;
\end{enumerate}
\end{multicols}
\item\label{it:submanifold} $\natural_i X(p_i,q_i)$ does not contain a smooth embedded submanifold diffeomorphic to $X(4,1)$, $X(9,2)$, $X(16,3)$, or $X(64,23)$.
\end{enumerate}
\end{thm}

Given a connected sum of lens spaces $L=\#_i L(p_i,q_i)$, one can ask about the minimal possible $b_2$ of a negative-definite filling of $L$. Such a $b_2$-minimal negative-definite filling naturally satisfies
$$0\leq b_2(X) \leq b_2(\natural_i X(p_i,q_i)).$$
Those $L$ which admit a filling with $b_2(X)=0$ were classified by Lisca \cite{Lisca:2007-1, Lisca:2007-2}. Theorem~\ref{thm:main} provides a classification of those $L$ at the other extreme, i.e.\ those for which the 4-manifold $\natural_i X(p_i,q_i)$ is a $b_2$-minimal negative-definite filling. Of course, for a general lens space establishing the size of the $b_2$-minimal negative-definite filling is highly nontrivial.

\begin{rem}
One should also note that Theorem~\ref{thm:main} really does require that the configurations in condition \ref{it:combinatorial} arise as {\em induced} subgraphs. For example the plumbing
\[\begin{tikzpicture}[xscale=1.0,yscale=1,baseline={(0,0)}]
    \node at (1-0.1, .4) {$-2$};
    \node at (2-0.1, .4) {$-2$};
    \node at (3-0.1, .4) {$-7$};
    \node (A1_1) at (1, 0) {$\bullet$};
    \node (A1_2) at (2, 0) {$\bullet$};
    \node (A1_3) at (3, 0) {$\bullet$};
    \path (A1_2) edge [-] node [auto] {$\scriptstyle{}$} (A1_3);
    \path (A1_1) edge [-] node [auto] {$\scriptstyle{}$} (A1_2);
  \end{tikzpicture}\]
  is a $b_2$-minimal negative-definite filling. In particular, the two vertices of weight $-2$ are adjacent and so do not induce a subgraph of type \ref{it:2-2}. On the other hand, the plumbing 
  \[\begin{tikzpicture}[xscale=1.0,yscale=1,baseline={(0,0)}]
    \node at (1-0.1, .4) {$-2$};
    \node at (2-0.1, .4) {$-7$};
    \node at (3-0.1, .4) {$-2$};
    \node (A1_1) at (1, 0) {$\bullet$};
    \node (A1_2) at (2, 0) {$\bullet$};
    \node (A1_3) at (3, 0) {$\bullet$};
    \path (A1_2) edge [-] node [auto] {$\scriptstyle{}$} (A1_3);
    \path (A1_1) edge [-] node [auto] {$\scriptstyle{}$} (A1_2);
     \end{tikzpicture}\]
    is not a $b_2$-minimal negative-definite filling because it contains an induced subgraph of type \ref{it:2-2}.
\end{rem}

\begin{rem}
If a smooth 4-manifold $X$ contains a smooth embedded submanifold of the form $X(p^2,pq-1)$ for $p,q$ relatively prime, then one can perform a \emph{generalized rational blow-down} to reduce $b_2(X)$ \cite{Park97blowdown}. Since the manifolds $X(4,1)$, $X(9,2)$, $X(16,3)$, and $X(64,23)$ appearing in Theorem~\ref{thm:main}\ref{it:submanifold} are all of this form, Theorem~\ref{thm:main} implies that if $\natural_i X(p_i,q_i)$ is not a $b_2$-minimal negative-definite filling, then it admits a generalized rational blow-down.
Thus Theorem~\ref{thm:main} shows that $\natural_i X(p_i,q_i)$ is a $b_2$-minimal negative-definite filling if and only if it does not admit any generalized rational blow-down.
\end{rem}

\begin{rem} The list provided by Theorem~\ref{thm:main}\ref{it:combinatorial} is the unique minimal list of forbidden configurations guaranteeing that canonical negative-definite plumbing is a $b_2$-minimal negative-definite filling. Firstly none of the configurations in the list are redundant, since no configuration on the list is contained in another one on the list. Furthermore all the graphs in Theorem~\ref{thm:main}\ref{it:combinatorial} must be included on any such list of forbidden configurations, since all their proper subplumbings are themselves $b_2$-minimal negative-definite fillings.\end{rem}

\subsection{Embeddings in 4-manifolds}
As any 3-manifold smoothly embeds in $S^5$~\cite{Hirsch:1961-1,Rohlin:1965-1,Wall:1965-1} and, hence, in any 5-manifold, it is natural to investigate the existence of embeddings of 3-manifolds in 4-manifolds. We will restrict ourselves to separating embeddings.\footnote{For many commonly considered 4-manifolds, specifically those with $b_1(X)=0$, there is no loss of generality in considering only separating embeddings. If a $4$-manifold $X$ contains an embedded closed 3-manifold $Y$ that is not separating, then one can glue together infinitely many copies of $X$ cut along $Y$ to construct an infinite cover of $X$ with deck group $\Z$. Such a covering corresponds to a nontrivial homomorphism 
$\pi_1(X)\rightarrow \Z$. As this homomorphism necessarily factors through $H_1(X;\Z)$, we have $b_1(X)\geq 1$.} As a particularly simple class of 3-manifolds, the embeddings of lens spaces (and connected sums thereof) have often been studied. For example, it is known that no lens space embeds in $S^4$~\cite{Hantzsche:1938-1} and the connected sums of lens spaces smoothly embedding in $S^4$ have been classified by Donald~\cite{Donald:2015-1} (see also \cite{Kawauchi-Kojima:1980-1,Gilmer-Livingston:1983-1,Fintushel-Stern:1987-1}). Embeddings of lens spaces in $\mathbb{C P}^2$ have also been studied (see, for example, \cite{Owens}).

\begin{ques}\label{q:embedding}
Is there a smooth closed 4-manifold $X$ such that every lens space smoothly embeds in $X$ as a separating submanifold?
\end{ques}

This question is somewhat subtle in several ways. The smooth hypothesis is important; Edmonds has shown that every lens space embeds topologically locally flatly into $\#_4 S^2\times S^2, \#_8 \mathbb{CP}^2$, $\#_5 \mathbb{CP}^2 \# \overline{\mathbb{CP}}^2$, and $\#_2 \mathbb{CP}^2 \#_2 \overline{\mathbb{CP}}^2$ \cite[Theorem 1.2]{Edmonds:2005-1}.\footnote{Although not explicitly stated by Edmonds, the case of $\#_5 \mathbb{CP}^2 \# \overline{\mathbb{CP}}^2$ also follows from the same methods.} Moreover, every punctured lens space smoothly embeds in $S^2\times S^2$ \cite[Proposition 2.3]{Edmonds-Livingston:1996-1}.

As evidence for a negative answer to Question~\ref{q:embedding}, we have the following result.
\begin{thm}\label{thm:embedding}
If $X$ is a smooth closed 4-manifold such that every lens space smoothly embeds in $X$ as a separating manifold, then $X$ is not spin and $$|\sigma(X)|\leq b_2(X)-4.$$ 
In particular,  for any positive integer $n$, not all lens spaces smoothly embed in $\#_n\mathbb{CP}^2$ or $\#_n\mathbb{CP}^2\#\overline{\mathbb{CP}}^2$.
\end{thm}
The condition the $X$ is not spin is straightforward. It was shown in~\cite{Aceto-Golla-Larson:2017-1} using the 10/8-Theorem~\cite{Furuta:2001-1}  that there is no positive integer $n$ such that every lens space smoothly embeds in $\#_n S^2\times S^2$. Their argument generalises easily to imply that there is no smooth spin 4-manifold that contains every lens space as a separating smooth submanifold (see Proposition~\ref{prop:spin}). The contribution of this paper is to establish the signature bound. Using Theorem~\ref{thm:main}, one can quickly establish that no smooth definite 4-manifold can contain all lens spaces as smoothly embedded separating submanifolds. Establishing the inequality $|\sigma(X)|\leq b_2(X)-4$ is much more delicate. In order to prove it we exhibit examples of lens spaces which do not admit small negative-definite fillings with either orientation.
\begin{thm}\label{thm:no_small_fillings}
For any positive integer $k$, there is a lens space $L_k$ such that if $X$ is a negative-definite filling of $L_k$ or $-L_k$, then $b_2(X)\geq k$.
\end{thm}

Given Theorem~\ref{thm:embedding} and the results of Edmonds \cite{Edmonds:2005-1}, it is natural to ask whether every lens space smoothly embeds in $\#_2 \mathbb{CP}^2 \#_2 \overline{\mathbb{CP}}^2$; we know that every lens space  topologically locally flatly embeds, but are unable to rule out the existence of a smooth embedding for any lens space.

\subsection{Comparison with the analytic and the symplectic settings}
There are also some links between Theorem~\ref{thm:main} and the theories of symplectic fillings\footnote{In the ensuing discussion, symplectic filling refers to the notion that is sometimes called a weak symplectic filling.} for lens spaces and smoothings of cyclic quotient singularities. The symplectic fillings for the standard (i.e.\ universally tight) contact structures on lens spaces were classified by Lisca~\cite{Lisca:2008-1}. Recently, Etnyre-Roy and Christian-Li extended this to classify the symplectic fillings of all tight contact structures on lens spaces~\cite{Etnyre-Roy:2021-1, Christian-Li:2020-1} (see also~\cite{Plamenevskaya,Kaloti,Fossati:2020-1}). This classification shows that every symplectic filling of a virtually overtwisted contact structure on $L(p,q)$ is diffeomorphic to a symplectic filling of the standard contact structure on $L(p,q)$.

If $L(p,q)$ is a lens space such that every vertex in the canonical negative-definite plumbing graph has weight $a_i\geq 5$, then the standard, and consequently, any, tight contact structure on $L(p,q)$ admits a unique minimal symplectic filling up to diffeomorphism (these are precisely the examples in Corollary~1.2(b) of \cite{Lisca:2008-1}). Theorem~\ref{thm:main} allows us to generalize this family of lens spaces with a unique minimal symplectic fillings. Moreover, since there is a one-to-one correspondence between the minimal symplectic fillings of lens spaces and the smoothings of cyclic quotient singularities~\cite{NemethiPopescu}, we also obtain the following:

\begin{cor}\label{cor:rigidsymplensspace}
Let $L(p,q)$ be a lens space as in Theorem~\ref{thm:main}\ref{it:combinatorial}. If $\xi$ is a tight contact structure on $L(p,q)$, then any minimal symplectic filling of $(L(p,q),\xi)$ is diffeomorphic to $X(p,q)$. Moreover, the base space of the miniversal deformation of the cyclic quotient singularity corresponding to $L(p,q)$ has a unique component.
\end{cor}
\begin{proof}
Suppose $X$ is a minimal symplectic filling of $(L(p,q),\xi)$. By \cite[Theorem~1.1]{Etnyre-Roy:2021-1}, we may assume that $X$ is a minimal filling for the standard contact structure on $L(p,q)$. Every symplectic filling of a lens space is negative-definite~\cite{Schonenberger, Etnyre:2004-1}, so Theorem~\ref{thm:main}, shows that $Q_{X}\cong Q_{X(p,q)}\oplus\langle-1\rangle^n$ for some $n\geq 0$. However, since $X$ is minimal, Lisca's classification~\cite[Theorem 1.1]{Lisca:2008-1} shows that $X$ is diffeomorphic to $X(p,q)$. The conclusion about the miniversal deformation follows from the one-to-one correspondence given by~\cite{NemethiPopescu}.
\end{proof}

Unfortunately, there are many lens spaces which do not satisfy Theorem~\ref{thm:main}\ref{it:combinatorial} and yet only have one minimal symplectic filling up to diffeomorphism for any tight contact structure (e.g.\ $L(n,n-1)$ with $n\geq 4$ \cite{Lisca:2008-1}). Even though Lisca's work provides a method to construct all symplectic fillings of a given lens space, it seems combinatorially challenging to obtain a description of all the lens spaces that admit a unique minimal symplectic filling up to diffeomorphism. Therefore, we pose this as a question:
\begin{ques}\label{question:rigidsymplensspace}
Is it possible to characterize all lens spaces $L(p,q)$ such that a minimal symplectic filling of $(L(p,q),\xi)$ for any tight contact structure $\xi$ is diffeomorphic to $X(p,q)$?
\end{ques}

\subsection{Proof outline for Theorem~\ref{thm:main}}
Theorem~\ref{thm:main} is proven via the chain of implications 
\[\mathrm{\ref{it:rigid}} \Rightarrow \mathrm{\ref{it:min_filling}} \Rightarrow \mathrm{\ref{it:submanifold}} \Rightarrow \mathrm{\ref{it:combinatorial}} \Rightarrow \mathrm{\ref{it:rigid}}\]
with \ref{it:combinatorial}~$\Rightarrow$~\ref{it:rigid} being the most difficult step.
The implication \ref{it:rigid} $\Rightarrow$ \ref{it:min_filling} holds, since the intersection form $Q_{\Xsum}$ is defined on a group of rank $b_2(\Xsum)$. The implication \ref{it:min_filling} $\Rightarrow$ \ref{it:submanifold} is obtained by noting that $L(4,1)$, $L(9,2)$, $L(16,3)$, and $L(64,23)$ bound rational homology balls, so a copy of $X(4,1)$, $X(9,2)$, $X(16,3)$, or $X(64,23)$ in $\Xsum$ can be cut out and replaced by a rational homology ball to reduce the second Betti number. In order to obtain \ref{it:submanifold} $\Rightarrow$ \ref{it:combinatorial}, we show that if the plumbing graph for $\Xsum$ contains one of the induced subgraphs listed in \ref{it:combinatorial}, then we can use the corresponding spherical generators to construct an embedded copy of  $X(4,1)$, $X(9,2)$, $X(16,3)$, or $X(64,23)$ in $\Xsum$.

To prove \ref{it:combinatorial}~$\Rightarrow$~\ref{it:rigid}, we invoke Donaldson's diagonalization theorem~\cite{Donaldson:1987-1}. Let $L=\#_i L(p_i,q_i)$ be a connected sum of lens spaces and $X$ be a smooth negative-definite filling of $L$ with intersection form $Q_X$.  Since $-L = - \#_i L(p_i,q_i) \cong \#_i L(p_i,p_i- q_i)$, it admits a canonical negative-definite filling $\natural_i X(p_i,p_i-q_i)$. Therefore we can form a smooth closed negative-definite 4-manifold
\[
X':=X \cup_{L} \natural_i X(p_i,p_i-q_i).
\]
By Donaldson's theorem~\cite{Donaldson:1987-1}, its intersection form $Q_{X'}$ is the standard diagonal form. The inclusions $X, \natural_i X(p_i,p_i-q_i)\hookrightarrow  X'$
induce a morphism of integral lattices
$$Q_X \oplus Q_{\natural_i X(p_i,p_i-q_i)}\hookrightarrow \langle-1\rangle^n,$$
where $n= b_2(X')$. Moreover, it can be checked that $Q_X$ is isomorphic to the orthogonal complement of the image of $Q_{\natural_i X(p_i,p_i-q_i)}$ in $\langle-1\rangle^n$. In particular, if $Q_{\natural_i X(p_i,p_i-q_i)}$ admits a unique embedding in $\langle-1\rangle^n$ up to automorphisms, then the intersection form $Q_X$ is uniquely determined. Therefore, it is enough to show that such an embedding is unique whenever $L$ satisfies \ref{it:combinatorial}. Most of the article is devoted to proving this last step. 

\subsection*{Structure of the article} In Section~\ref{sec:topology}, we prove the easy steps of Theorem~\ref{thm:main} (i.e.\ \ref{it:rigid} $\Rightarrow$ \ref{it:min_filling} $\Rightarrow$ \ref{it:submanifold} $\Rightarrow$ \ref{it:combinatorial}) and reduce \ref{it:combinatorial}~$\Rightarrow$~\ref{it:rigid} to a lattice theoretical statement (see Theorem~\ref{thm:technical}). We then deduce Theorem~\ref{thm:main} assuming Theorem~\ref{thm:technical}. In Section~\ref{Section:preliminaries} and Section~\ref{sec:lattice_analysis}, we prove Theorem~\ref{thm:technical} with the bulk of the technical analysis  being contained in Section~\ref{sec:lattice_analysis}. In Section~\ref{sec:no_small}, we construct some examples of lens spaces that do not admit any small definite fillings of either sign. Finally, in Section~\ref{sec:embeddings}, we use these examples to prove Theorem~\ref{thm:embedding}. 

\subsection*{Notation and conventions}
In this article, every 3-manifold is smooth, connected, closed, and oriented. All 4-manifolds are smooth, connected, compact, and oriented. We indicate with $-Y$ the manifold $Y$ with reversed orientation. For two manifolds $Y_1$ and $Y_2$, the symbol $Y_1 \cong Y_2$ is used to indicate that there is an orientation preserving diffeomorphism between $Y_1$ and $Y_2$. Unless explicitly otherwise stated, homology in this article is homology with integral coefficients.

\subsection*{Acknowledgements}
We are grateful to John Etnyre for useful conversations and his interest. We thank also several anonymous referees for their many insightful comments and suggestions.
PA is supported by the European Union’s Horizon 2020 research and innovation programme under the Marie Sk\l odowska-Curie action, Grant No.\ 101030083 LDTSing.
DM is partly supported by NSERC and FRQNT grants. 
JP is partially supported by Samsung Science and Technology Foundation (SSTF-BA2102-02) and the POSCO TJ Park Science Fellowship.

\section{Topology and the proof of Theorem~\ref{thm:main}}\label{sec:topology}
In this section, we discuss the 4-manifold topology necessary for Theorem~\ref{thm:main} and give the proof modulo a key technical result, Theorem~\ref{thm:technical}, whose proof will occupy Section~\ref{Section:preliminaries} and  Section~\ref{sec:lattice_analysis}.

An \emph{integral lattice} is a pair $(L, Q_L)$, where $L$ is a free abelian group and $Q_L : L \times L \rightarrow \Z$ is a symmetric bilinear pairing. A \emph{map of lattices} is a linear map
\[
\phi: (L, Q_L) \rightarrow (L', Q_{L'})
\]
 preserving the bilinear pairings. In particular, the intersection form of a compact oriented 4-manifold $X$ defines a lattice $(H_2(X)/ \tors , Q_X)$, which we will often refer to as just $Q_X$.
 
Recall that any connected sum of lens spaces $L=\Lsum$ bounds the boundary connected sum of canonical plumbings $X=\Xsum$. The intersection form of $X$ will be discussed in further detail in Section~\ref{Section:preliminaries}. However, for the purposes of this section the following well-known lemma contains all the necessary properties that we need for this section. Its proof makes use of a stronger statement (Lemma~\ref{lem:def_calc}) that will be proven later.
 
 \begin{lem}\label{lem:length_one}
If $X=\Xsum$ is the boundary connected sum of canonical plumbings, then for any $x\in H_2(X)$ with $x\neq 0$, we have $Q_X(x,x)\leq -2$. In particular, $X$ is negative-definite and the intersection form does not contain any elements with $Q_X(x,x)=-1$.
\end{lem}
\begin{proof}
The intersection form of $-X$, that is, $X$ with the opposite orientation, is isomorphic to a lattice associated to a weighted plumbing graph in which every vertex has weight greater than or equal to two. Lemma~\ref{lem:def_calc} applied to such a plumbing graph shows that $Q_{-X}(x,x)=-Q_X(x,x)\geq 2$ for every nonzero $x\in H_2(X)$.
\end{proof}

\subsection{Small fillings} 
Next we perform the necessary constructions to establish the implications \ref{it:min_filling} $\Rightarrow$ \ref{it:submanifold} and  \ref{it:submanifold} $\Rightarrow$ \ref{it:combinatorial} in Theorem~\ref{thm:main}.

\begin{lem}\label{lem:rational_blowdowns}
Let $L=\Lsum$ be a connected sum of lens spaces bounding a smooth negative-definite manifold $X$. If $X$ contains an embedded submanifold diffeomorphic to one of $X(4,1)$, $X(9,2)$, $X(16,3)$, or $X(64,23)$, then $L$ bounds a smooth negative-definite 4-manifold $X'$ with $b_2(X')<b_2(X)$.
\end{lem}
\begin{proof}
We note that the lens spaces $L(4,1)$, $L(9,2)$, $L(16,3)$, and $L(64,23)$ all bound rational homology balls (see e.g.\ \cite[Theorem 1.2]{Lisca:2007-1}). The lemma follows from the more general observation that if $X$ contains an embedded submanifold $A$ diffeomorphic to $X(r,s)$, where $\partial X(r,s)\cong L(r,s)$ bounds a rational homology ball $B$, then we can form the smooth manifold
\[
X'= \left(X \smallsetminus \mathrm{int} A\right) \cup_{-L(r,s)} B.
\]
Since $L(r,s)$ is a rational homology sphere, the Mayer-Vietoris exact sequence with rational coefficients quickly shows that $b_2(X')=b_2(X)-b_2(X(r,s))$ and Novikov additivity shows that $\sigma(X')=\sigma(X) -\sigma(X(r,s))$. Since both $X$ and $X(r,s)$ are negative-definite, this shows that $X'$ is negative-definite with $b_2(X')<b_2(X)$.
\end{proof}

Let $F_1$ and $F_2$ be two smoothly embedded oriented surfaces intersecting transversely at a single point $p$ in a 4-manifold $X$. If $B$ is sufficiently small 4-ball around $p$, then $F_1$ and $F_2$ both intersect $B$ in a disk and the set $(F_1 \cup F_2)\cap \partial B$ is a Hopf link in $\partial B \cong S^3$. Since the Hopf link bounds an annulus in $B$ we may produce a connected embedded surface $F$ by replacing the two disks $F_1\cap B$ and $F_2\cap B$ by an annulus in $B$. 
This operation, which we refer to as resolving the intersection point between $F_1$ and $F_2$, will be used in the proof of the following lemma. An important observation is that the homology class represented by $F$ is 
\[[F]=[F_1]+[F_2]\in H_2(X;\Z).\]    

\begin{lem}\label{lem:constructing_submanifolds}
Let $L=\Lsum$ be a connected sum of lens spaces. If the canonical plumbing graph associated to $L$ contains one of the induced subgraphs in Theorem~\ref{thm:main}\ref{it:combinatorial}, then $\Xsum$ contains an embedded submanifold diffeomorphic to one of the plumbed 4-manifolds $X(4,1)$, $X(9,2)$, $X(16,3)$, or $X(64,23)$.
\end{lem}

\begin{proof}
The manifolds $X(4,1)$, $X(9,2)$, $X(16,3)$ and $X(64,23)$ are plumbed according to the linear plumbing graphs 
$$\begin{tikzpicture}[xscale=1.0,yscale=1,baseline={(0,0)}]
    \node at (1-.1, .4) {$-4$};
    \node (A_1) at (1, 0) {$\bullet$};
  \end{tikzpicture},   \qquad
  \begin{tikzpicture}[xscale=1.0,yscale=1,baseline={(0,0)}]
    \node at (1-0.1, .4) {$-5$};
    \node at (2-0.1, .4) {$-2$};
    \node (A_1) at (1, 0) {$\bullet$};
    \node (A_2) at (2, 0) {$\bullet$};
        \path (A_1) edge [-] node [auto] {$\scriptstyle{}$} (A_2);
  \end{tikzpicture},  \qquad
  \begin{tikzpicture}[xscale=1.0,yscale=1,baseline={(0,0)}]
    \node at (1-0.1, .4) {$-6$};
    \node at (2-0.1, .4) {$-2$};
    \node at (3-0.1, .4) {$-2$};
    \node (A_1) at (1, 0) {$\bullet$};
    \node (A_2) at (2, 0) {$\bullet$};
    \node (A_3) at (3, 0) {$\bullet$};
        \path (A_1) edge [-] node [auto] {$\scriptstyle{}$} (A_2);
    \path (A_2) edge [-] node [auto] {$\scriptstyle{}$} (A_3);
  \end{tikzpicture},
  $$
  and
  $$
  \begin{tikzpicture}[xscale=1.0,yscale=1,baseline={(0,0)}]
    \node at (1-0.1, .4) {$-3$};
    \node at (2-0.1, .4) {$-5$};
    \node at (3-0.1, .4) {$-3$};
    \node at (4-0.1, .4) {$-2$};
    \node (A1_1) at (1, 0) {$\bullet$};
    \node (A1_2) at (2, 0) {$\bullet$};
    \node (A1_3) at (3, 0) {$\bullet$};
    \node (A1_4) at (4, 0) {$\bullet$};
    \path (A1_2) edge [-] node [auto] {$\scriptstyle{}$} (A1_3);
    \path (A1_3) edge [-] node [auto] {$\scriptstyle{}$} (A1_4);
    \path (A1_1) edge [-] node [auto] {$\scriptstyle{}$} (A1_2);
  \end{tikzpicture},$$
  respectively. Thus the statement of the lemma is clearly true if the canonical plumbing graph contains one of the configurations \ref{it:4}, \ref{it:52}, \ref{it:622}, or \ref{it:3532}.

If the induced subgraph \ref{it:2-2} is contained in the plumbing graph, then we obtain an embedded sphere with self-intersection $-4$ by tubing together the two disjoint spheres with self-intersection $-2$. A tubular neighbourhood of this $-4$-framed sphere is a copy of $X(4,1)$.
  
If the induced subgraph \ref{it:3-22} is contained in the plumbing graph, then we obtain a copy of $X(9,2)$ as a submanifold. Tubing the sphere of self-intersection $-3$ with one of the spheres of self-intersection $-2$ yields a sphere of self-intersection $-5$ that intersects the other sphere of self-intersection $-2$ in a single point. A tubular neighbour of the union of these two spheres of self-intersection $-5$ and $-2$ gives the desired copy of $X(9,2)$.

In cases \ref{it:33}, \ref{it:323}, and \ref{it:3223}, we obtain an embedded copy of $X(4,1)$. In each case, this can be obtained by taking a copy of each of the spherical generators and resolving the intersection points to obtain an embedded sphere of self-intersection $-4$.

Finally, for \ref{it:2235}, we obtain an embedded copy of $X(16,3)$. An embedded sphere with self-intersection $-6$ is obtained by resolving the intersection point between the spheres of self-intersection $-3$ and $-5$. A tubular neighbourhood of this sphere along with tubular neighbourhoods of the two $-2$-framed spheres gives the copy of $X(16,3)$. 
\end{proof}

\subsection{Rigidity and the intersection form}
In this section, we discuss the intersection forms of definite fillings. 

\begin{lem}\label{lem:orthogonal_complement}
Let $Y$ be a closed oriented rational homology sphere which is the boundary of compact 4-manifolds $P$ and $X$, where $H_1(P)=0$. If $W$ is the closed oriented 4-manifold $W= P\cup_Y -X$, then the inclusions  $P, -X \hookrightarrow W$ induce a map of lattices 
\[\iota: Q_P \oplus -Q_X  \rightarrow Q_W\]
such that $-Q_X \cong \iota(Q_P)^\bot\subseteq Q_W$, where $\iota(Q_P)^\bot$ is the orthogonal complement of $\iota(Q_P)$ in $Q_W$. \end{lem}

\begin{proof}
Using a portion of the long exact Mayer-Vietoris sequence, we obtain the exact sequence
\[
0=H_2(Y) \rightarrow H_2(P) \oplus H_2(-X) \rightarrow H_2(W).
\]
Since the maps induced by inclusion preserve the intersection pairing, this gives an embedding of lattices $\iota: Q_P \oplus Q_{-X} \rightarrow Q_W$ such that the image has full rank in $Q_W$. This implies that $\iota\left(Q_{-X}\right)$ has finite index in $\iota(Q_P)^\bot$. Consider now the exact sequence 
\[
H_2(-X) \rightarrow H_2(W) \rightarrow H_2(W,-X)
\]
Note that $H_2(W,-X)$ is torsion-free since it is isomorphic to $H_2(P, Y) \cong H^2(P)$ and the torsion subgroup of $H^2(P)$ is isomorphic to the torsion subgroup of $H_1(P)=0$.  Thus the quotient of $H_2(W)$ by the image of $H_2(-X)$ is torsion-free. However the quotient of $H_2(W)$ by the image of $H_2(-X)$ contains a subgroup isomorphic to the quotient $\iota(Q_P)^\bot/\iota(Q_{-X})$ and so we conclude that $\iota(Q_{-X})=\iota(Q_P)^\bot$. Since $\iota\left(Q_{-X}\right)$ is isomorphic to $-Q_X$, the lemma follows.
\end{proof}

The following concept will play a key role in proving Theorem~\ref{thm:main}.
\begin{defn}\label{def:rigidity} A lattice $(L, Q_L)$ is {\em rigid}, if for any positive integer $N$ and any pair of lattice maps $\phi, \phi': (L, Q_L) \rightarrow (\Z^N, \langle 1\rangle^N)$, there is an automorphism $\alpha \in \aut (\Z^N)$ such that $\phi'= \alpha \circ \phi$.
\end{defn}

The relevance of this definition is that it gives a criterion for proving that for certain 3-manifolds the intersection form of any negative-definite filling is unique up to stabilisation with $\langle -1\rangle$ summands.

\begin{prop}\label{prop:rigidity_application}
If $Y$ is a rational homology sphere which bounds a smooth positive-definite 4-manifold $P$ with $H_1(P)=0$ such that the intersection form of $P$ is rigid, then for any pair of smooth negative-definite fillings $X$ and $X'$ with $b_2(X) \geq b_2(X')$, we have $$Q_X \cong Q_{X'} \oplus \langle-1 \rangle^{b_2(X)-b_2(X')}.$$
\end{prop}

\begin{proof}
The closed manifolds $W= P\cup -X$ and $W'=P\cup -X'$ are smooth and positive-definite and hence diagonalizable by Donaldson's theorem. Let $n=b_2(W)=b_2(X)+b_2(P)$ and $n'=b_2(W')= b_2(X')+b_2(P)$. Note that $n-n' = b_2(X) - b_2(X')$. Thus by Lemma~\ref{lem:orthogonal_complement} we have embeddings 
\[\iota:Q_P\hookrightarrow (\Z^{n},\langle 1 \rangle^n) \qquad\text{ and }\qquad \iota':Q_P\hookrightarrow (\Z^{n'},\langle 1 \rangle^{n'})\]
such that the orthogonal complements give us isomorphisms 
$$-Q_X\cong \iota(Q_P)^\bot \qquad \text{ and } \qquad -Q_{X'}\cong \iota'(Q_P)^\bot.$$ 
By embedding $(\Z^{n'},\langle 1 \rangle^{n'})$ in $(\Z^{n},\langle 1 \rangle^n)$ in the natural way, we further get an embedding 
\[\iota'': Q_P\hookrightarrow (\Z^{n},\langle 1 \rangle^{n})\]
such that  $\iota''(Q_P)^\bot$ is isomorphic to $-Q_{X'} \oplus \langle 1 \rangle^{n-n'}$. However since  the intersection form of $P$ is rigid, there is an automorphism of $(\Z^n, \langle 1 \rangle^n)$ carrying $\iota''(Q_P)$ to $\iota(Q_P)$. This automorphism restricts to give isomorphism between the orthogonal complements.
\end{proof}

Any connected sum of lens spaces $L=\Lsum$ is the boundary of the manifold $-\natural_i X(p_i, p_i-q_i)$, which is positive-definite by Lemma~\ref{lem:length_one}. The key ingredient in the proof of the implication \ref{it:combinatorial}$\Rightarrow$\ref{it:rigid} of Theorem~\ref{thm:main} is the following result which will be established in Section~\ref{sec:lattice_analysis}.
\begin{thm}\label{thm:technical}
If $L=\Lsum$ is a connected sum of lens spaces such that the corresponding plumbing graph does not contain any of the configurations in Theorem~\ref{thm:main}\ref{it:combinatorial}, then the intersection form of $-\natural_i X(p_i, p_i-q_i)$ is a rigid lattice.
\end{thm}
Modulo the proof of Theorem~\ref{thm:technical}, we are ready to prove Theorem~\ref{thm:main}.
\begin{proof}[Proof of Theorem~\ref{thm:main}]
The implication \ref{it:rigid}$\Rightarrow$\ref{it:min_filling} is evident. The implication \ref{it:min_filling}$\Rightarrow$\ref{it:submanifold} is established via its contrapositive in Lemma~\ref{lem:rational_blowdowns}. The implication \ref{it:submanifold}$\Rightarrow$\ref{it:combinatorial} is established via its contrapositive in Lemma~\ref{lem:constructing_submanifolds}. Finally, we establish \ref{it:combinatorial}$\Rightarrow$\ref{it:rigid}. If $L=\Lsum$ is a connected sum of lens spaces satisfying the hypotheses of \ref{it:combinatorial}, then the intersection form of $-\natural_i X(p_i, p_i-q_i)$ is rigid by Theorem~\ref{thm:technical}. Thus Proposition~\ref{prop:rigidity_application} implies that any pair of smooth negative-definite fillings of $L$ have the same intersection form up to stabilising with $\langle -1 \rangle$ summands. By Lemma~\ref{lem:length_one}, $\Xsum$ is a smooth negative-definite filling whose intersection form does not contain any elements of length $-1$. Thus we see that every other possible intersection form must be obtained by stabilising the intersection form of $\Xsum$.
\end{proof}

\section{Canonical plumbings and rigidity}\label{Section:preliminaries}
In this section we explore the relationship between the intersection form of $\Xsum$ and the intersection form of $-\natural_i X(p_i, p_i-q_i)$. We begin with some definitions.

\begin{defn}\label{def:QP}
Let $P$ be a disjoint union of linear plumbing graphs and $V(P)$ be the set of its vertices. We define the \emph{lattice associated to $P$} to be the lattice
\[
(\Z V(P), Q_P),
\]
i.e.\ the group is the free abelian group generated by the vertices of $P$ and the pairing
$Q_P$ is defined by the rule that for all vertices $u,v \in V(P)$ we have
\[
Q_P(u,v)= \begin{cases}
w(v) & \text{if $u=v$}\\
-1 &  \text{if $u,v$ adjacent}\\
0 & \text{otherwise,}
\end{cases}
\]
where $w(v)$ is the weight of $v$.
For brevity, we often refer to $(\Z V(P), Q_P)$ simply as $Q_P$. \end{defn}

It will frequently be convenient to think in terms of adjusted weights which are defined as follows.

\begin{defn}\label{def:adjustedweight}
Let $P$ be a disjoint union of linear plumbing graphs. We define the {\em adjusted weight} of a vertex $v\in P$ to be
\[w'(v):=w(v)-d(v),\]
where $d(v)$ is the degree of $v$.
\end{defn}
\begin{notation}
When depicting a vertex in a plumbing graph we use a solid dot to indicate that the labelling is the weight $\begin{tikzpicture}[xscale=1.0,yscale=1,baseline={(0,0)}]
    \node at (1, .4) {$w(v)$};
    \node (A_1) at (1, 0) {$\bullet$};
  \end{tikzpicture}$ and a hollow dot to indicate that the labelling is the adjusted weight $\begin{tikzpicture}[xscale=1.0,yscale=1,baseline={(0,0)}]
    \node at (1, .4) {$w'(v)$};
    \node (A_1) at (1, 0) {$\circ$};
  \end{tikzpicture}$.
\end{notation}
For example, with this notational convention
$$\begin{tikzpicture}[xscale=1.0,yscale=1,baseline={(0,0)}]
    \node at (1, .4) {$1$};
    \node at (2, .4) {$2$};
    \node at (3, .4) {$2$};
    \node (A1_1) at (1, 0) {$\circ$};
    \node (A1_2) at (2, 0) {$\circ$};
    \node (A1_3) at (3, 0) {$\circ$};
    \path (A1_1) edge [-] node [auto] {$\scriptstyle{}$} (A1_2);
        \path (A1_2) edge [-] node [auto] {$\scriptstyle{}$} (A1_3);
  \end{tikzpicture}
  \qquad\text{ and }\qquad
  \begin{tikzpicture}[xscale=1.0,yscale=1,baseline={(0,0)}]
    \node at (1, .4) {$2$};
    \node at (2, .4) {$4$};
    \node at (3, .4) {$3$};
    \node (A1_1) at (1, 0) {$\bullet$};
    \node (A1_2) at (2, 0) {$\bullet$};
    \node (A1_3) at (3, 0) {$\bullet$};
    \path (A1_1) edge [-] node [auto] {$\scriptstyle{}$} (A1_2);
        \path (A1_2) edge [-] node [auto] {$\scriptstyle{}$} (A1_3);
  \end{tikzpicture}
  $$
  both depict the same plumbing graph.
  
If a 4-manifold $X$ is obtained by plumbing $D^2$-bundles over $S^2$ according to a weighted plumbing graph $P$, then the intersection form of $X$ is isomorphic to the lattice $(\Z V(P), Q_P)$, where the isomorphism sends a vertex to the corresponding spherical generator in $H_2(X)$.

The following lemma taken with $M=2$ implies Lemma~\ref{lem:length_one}.
\begin{lem}\label{lem:def_calc}
If $P$ is a disjoint union of linear plumbing graphs in which every vertex has weight $w(v)\geq M$ for some integer $M\geq 2$, then for any $x\in \Z V(P)$ with $x\neq 0$, we have $Q_P(x,x)\geq M$.
\end{lem}
\begin{proof}
If $P$ can be decomposed into connected components $P=\bigsqcup P_i$, then we have a decomposition of the lattice $Q_P = \bigoplus Q_{P_i}$. Thus it suffices to prove the lemma when $P$ is connected. In this case, suppose that we have vertices $v_1, \dots, v_n$, where these are ordered so that $v_i$ and $v_{i+1}$ are adjacent for $i=1, \dots, n-1$. By hypothesis, $w(v_i)\geq M$ for all $i$. Thus if we take $x= \sum_{i=1}^n c_i v_i$ for integers $c_i$, then completing the square reveals that for any $1\leq  j \leq n$ we have 
\begin{align*}
Q_P(x,x)&= \sum_{i=1}^{n} (w(v_i)-2) c_i^2 + \sum_{i=1}^{n-1} (c_i - c_{i+1})^2\\
&\geq (w(v_j)-2) c_j^2+c_1^2 + (c_1 - c_{2})^2 + \dots + (c_{n-1} - c_n)^2 + c_n^2\\
&\geq (w(v_j)-2) |c_j| +|c_1|+ |c_1 - c_{2}| + \dots + |c_{n-1} - c_n| + |c_n|\\
&\geq (w(v_j)-2) |c_j|+ 2|c_j|\\
&=w(v_j)|c_j|\\
&\geq M|c_j|
\end{align*}
where we applied the triangle inequality to show:
\[
|c_j|\leq |c_1|+ |c_1 - c_{2}| + \dots + |c_{j-1} - c_j|
\]
and
\[
|c_j|\leq |c_j - c_{j+1}| + \dots + |c_{n-1} - c_n| + |c_n|.
\]

If $x\neq 0$, then $c_j\neq 0$ for at least one $j$ and so we have the required bound.
\end{proof}

\subsection{The dual plumbing}
A connected sum of lens spaces $\Lsum$ arises as the boundary the plumbings $\Xsum$ and $-\natural_i X(p_i, p_i-q_i)$. If the manifold $\Xsum$ is obtained by plumbing according to a canonical plumbing graph $P$ and $-\natural_i X(p_i, p_i-q_i)$ is obtained by plumbing along a canonical plumbing graph $P^*$, then we say that $P^*$ is the {\em dual} of $P$. The task of this section is to relate $P$ and $P^*$ explicitly.

We will work with negative continued fractions and use the following notation
$$[a_1,\dots, a_n]^-= a_1 - \cfrac{1}{a_2 - \cfrac{1}{\ddots - \cfrac{1}{a_n}}}.$$

Here and throughout the article, we will use $a^{[m]}$ to denote the tuple 
\[
\underbrace{a, a, \dots,a, a}_{\text{$m$ times}},
\]
with the understanding that $a^{[0]}$ denotes the empty tuple. The Riemenschneider point rule \cite{Riemenschneider:1974-1} tells us that if $p/q$ has the continued fraction of the form
\[
\frac{p}{q}=\left[2^{[a_0]}, b_1, 2^{[a_1]}, \dots, b_k, 2^{[a_k]}\right]^-
\]
where the $a_i$ and $b_i$ are integers satisfying $a_i\geq 0$ and $b_i\geq 3$, then $p/(p-q)$ takes the form
\begin{equation}\label{eq:Riemanschneider_long}
\frac{p}{p-q}=\begin{cases}
\left[a_0+1\right]^- &\text{if $k=0$}\\
\left[a_0+2, 2^{[b_1-3]}, a_1+3, 2^{[b_2-3]}, \dots,  2^{[b_k-3]}, a_k + 2\right]^- &\text{if $k\geq 1$}.
\end{cases}
\end{equation}

When recast in terms of plumbings \eqref{eq:Riemanschneider_long} gives a nice relationship between the weights of $P$ and the adjusted weights of $P^*$.
\begin{lem}\label{lem:duality}
Suppose that a plumbing graph $P$ contains a linear component of the form
$$\begin{tikzpicture}[xscale=1.0,yscale=1,baseline={(0,0)}]
    \node at (1-0.1, .4) {$-c_1$};
    \node at (2-0.1, .4) {$-c_2$};
        \node at (3-0.1, .4) {$-c_3$};
    \node at (5-0.1, .4) {$-c_{n}$};
    \node (A1_1) at (1, 0) {$\bullet$};
    \node (A1_2) at (2, 0) {$\bullet$};
    \node (A1_3) at (3, 0) {$\bullet$};
    \node (A1_4) at (4, 0) {$\cdots$};
    \node (A1_5) at (5, 0) {$\bullet$};
    \path (A1_2) edge [-] node [auto] {$\scriptstyle{}$} (A1_3);
    \path (A1_3) edge [-] node [auto] {$\scriptstyle{}$} (A1_4);
        \path (A1_4) edge [-] node [auto] {$\scriptstyle{}$} (A1_5);
    \path (A1_1) edge [-] node [auto] {$\scriptstyle{}$} (A1_2);
  \end{tikzpicture},$$
  where the $c_i$ are integers satisfying
  \[
  (c_1, \dots, c_n)= \left(2^{[a_0]}, b_1, 2^{[a_1]}, \dots, b_k, 2^{[a_k]}\right),
  \]
for $a_i$ and $b_i$ integers satisfying $a_i\geq 0$ and $b_i\geq 3$. Then the corresponding component of $P^*$ is of the form
  $$\begin{tikzpicture}[xscale=1.0,yscale=1,baseline={(0,0)}]
    \node at (1-0, .4) {$d_1$};
    \node at (2-0, .4) {$d_2$};
        \node at (3-0, .4) {$d_3$};
    \node at (5-0, .4) {$d_{n'}$};
    \node (A1_1) at (1, 0) {$\circ$};
    \node (A1_2) at (2, 0) {$\circ$};
    \node (A1_3) at (3, 0) {$\circ$};
    \node (A1_4) at (4, 0) {$\cdots$};
    \node (A1_5) at (5, 0) {$\circ$};
    \path (A1_2) edge [-] node [auto] {$\scriptstyle{}$} (A1_3);
    \path (A1_3) edge [-] node [auto] {$\scriptstyle{}$} (A1_4);
        \path (A1_4) edge [-] node [auto] {$\scriptstyle{}$} (A1_5);
    \path (A1_1) edge [-] node [auto] {$\scriptstyle{}$} (A1_2);
  \end{tikzpicture},$$
  where
  \[
 \left(d_1, \dots, d_{n'}\right) =\left(a_0+1, 0^{[b_1-3]}, a_1+1, \dots, 0^{[b_k-3]}, a_k+1\right).
  \]
\end{lem}
\begin{proof}
According to \eqref{eq:Riemanschneider_long} the corresponding component of $P^*$ is of the form
$$\begin{tikzpicture}[xscale=1.0,yscale=1,baseline={(0,0)}]
    \node at (1-0, .4) {$\widetilde{d}_1$};
    \node at (2-0, .4) {$\widetilde{d}_2$};
    \node at (3-0, .4) {$\widetilde{d}_3$};
    \node at (5-0, .4) {$\widetilde{d}_{n'}$};
    \node (A1_1) at (1, 0) {$\bullet$};
    \node (A1_2) at (2, 0) {$\bullet$};
    \node (A1_3) at (3, 0) {$\bullet$};
    \node (A1_4) at (4, 0) {$\cdots$};
    \node (A1_5) at (5, 0) {$\bullet$};
    \path (A1_2) edge [-] node [auto] {$\scriptstyle{}$} (A1_3);
    \path (A1_3) edge [-] node [auto] {$\scriptstyle{}$} (A1_4);
        \path (A1_4) edge [-] node [auto] {$\scriptstyle{}$} (A1_5);
    \path (A1_1) edge [-] node [auto] {$\scriptstyle{}$} (A1_2);
  \end{tikzpicture},$$
  where $$(\widetilde{d}_1, \dots, \widetilde{d}_{n'})=(a_0+1)$$ if $k=0$ and
  $$(\widetilde{d}_1, \dots, \widetilde{d}_{n'})=(a_0+2, 2^{[b_1-3]}, a_1+3, 2^{[b_2-3]}, \dots,  2^{[b_k-3]}, a_k + 2)
  $$
  if $k\geq 1$.
Subtracting the degree of the vertex from each of these weights produces the tuple of adjusted weights in the statement of the proposition. 
\end{proof}

This allows us to recast the conditions of Theorem~\ref{thm:main}\ref{it:combinatorial} in terms of the properties of the dual. More precisely, the conditions of Theorem~\ref{thm:main}\ref{it:combinatorial} translate to the following list of conditions on the dual, which we will refer to as the Working Conditions.
\begin{conditions} The Working Conditions for $\Gamma$, a disjoint union of linear plumbing graphs, are the following:
\begin{enumerate}[label=\Roman*]
\item\label{it:positivity} every vertex of $\Gamma$ satisfies $w(v)\geq 2$;
\item\label{it:largeweight} every vertex of $\Gamma$ satisfies $w'(v)\leq 3$ and $\Gamma$ contains at most one vertex $v$ with $w'(v) >1$;
\item\label{it:3adjacent}  the graph $\Gamma$ does not contain three adjacent vertices with $w'(v)>0$;
\item\label{it:weight3condition} if $\Gamma$ contains a vertex $v$ with $w'(v)=3$, then $\Gamma$ does not contain a subgraph of the form $$\begin{tikzpicture}[xscale=1.0,yscale=1,baseline={(0,0)}]
    \node at (1, .4) {$1$};
    \node at (2.0, .4) {$1$};
    \node (A1_1) at (1, 0) {$\circ$};
    \node (A1_2) at (2, 0) {$\circ$};
    \path (A1_1) edge [-] node [auto] {$\scriptstyle{}$} (A1_2);
  \end{tikzpicture};$$
\item\label{it:long_chain} if $\Gamma$ contains a chain of adjacent vertices $v_0, \dots, v_{k+1}$ where 
$$\text{$w'(v_0)\geq 1$, $w'(v_1)=\dots = w'(v_{k})=0$, $w'(v_{k+1})=1$ and $k \neq 0$,}$$
then $k\geq w'(v_0)+1$.
\item\label{it:forbidden_configs} the graph $\Gamma$ does not contain any subgraphs of the following forms:

\begin{enumerate}[label=(\alph*),font=\upshape]
\item\label{it:110012} $\begin{tikzpicture}[xscale=1.0,yscale=1,baseline={(0,0)}]
    \node at (-2.0+0.0, .4) {$1$};
    \node at (-1.0+0.0, .4) {$1$};
    \node at (0.0, .4) {$0$};
    \node at (1.0, .4) {$0$};
    \node at (2.0, .4) {$1$};
    \node at (3.0, .4) {$2$};
        \node (A1_8) at (-2, 0) {$\circ$};
        \node (A1_9) at (-1, 0) {$\circ$};
    \node (A1_0) at (0, 0) {$\circ$};
    \node (A1_1) at (1, 0) {$\circ$};
    \node (A1_2) at (2, 0) {$\circ$};
    \node (A1_3) at (3, 0) {$\circ$};
            \path (A1_8) edge [-] node [auto] {$\scriptstyle{}$} (A1_9);
        \path (A1_9) edge [-] node [auto] {$\scriptstyle{}$} (A1_0);
    \path (A1_0) edge [-] node [auto] {$\scriptstyle{}$} (A1_1);
    \path (A1_1) edge [-] node [auto] {$\scriptstyle{}$} (A1_2);
        \path (A1_2) edge [-] node [auto] {$\scriptstyle{}$} (A1_3);
  \end{tikzpicture}$,
\item\label{it:31001} $\begin{tikzpicture}[xscale=1.0,yscale=1,baseline={(0,0)}]
    \node at (-1.0, .4) {$3$};
    \node at (0.0, .4) {$1$};
    \node at (1.0, .4) {$0$};
    \node at (2.0, .4) {$0$};
    \node at (3.0, .4) {$1$};
        \node (A1_9) at (-1, 0) {$\circ$};
    \node (A1_0) at (0, 0) {$\circ$};
    \node (A1_1) at (1, 0) {$\circ$};
    \node (A1_2) at (2, 0) {$\circ$};
    \node (A1_3) at (3, 0) {$\circ$};
        \path (A1_9) edge [-] node [auto] {$\scriptstyle{}$} (A1_0);
    \path (A1_0) edge [-] node [auto] {$\scriptstyle{}$} (A1_1);
    \path (A1_1) edge [-] node [auto] {$\scriptstyle{}$} (A1_2);
        \path (A1_2) edge [-] node [auto] {$\scriptstyle{}$} (A1_3);
  \end{tikzpicture}$.
\end{enumerate}
\end{enumerate}
\end{conditions}

\begin{lem}\label{lem:forbidden_configurations}
Let $P$ be a disjoint union of linear plumbing graphs with $w(v)\leq -2$ for all vertices. If $P$ does not contain any of the configurations listed in Theorem~\ref{thm:main}\ref{it:combinatorial}, then the dual $P^*$ satisfies the Working Conditions.
\end{lem}
\begin{proof}
Note that Condition~\ref{it:positivity} is implied by the definition of the dual $P^*$. The content of the lemma is in establishing the remaining five conditions. The strategy is to go through each of these conditions in turn and use Lemma~\ref{lem:duality} to show that if they fail then $P$ would contain one of the subgraphs listed in Theorem~\ref{thm:main}\ref{it:combinatorial}.

If $P^*$ contains a vertex $v$ with $w'(v)=n>1$, then Lemma~\ref{lem:duality} shows that $P$ contains a chain of $n-1$ vertices of weight $-2$. Since $P$ cannot contain two nonadjacent vertices of weight $-2$, this implies that $n\leq 3$. Moreover if $P^*$ contains two vertices with $w'(v)\geq 2$, then this would give rise to two nonadjacent vertices of weight $-2$ in $P$. This establishes Condition~\ref{it:largeweight}.

Suppose that $P^*$ contains a subgraph of the form  
$$\begin{tikzpicture}[xscale=1.0,yscale=1,baseline={(0,0)}]
    \node at (1, .4) {$a$};
    \node at (2, .4) {$b$};
    \node at (3, .4) {$c$};
    \node (A1_1) at (1, 0) {$\circ$};
    \node (A1_2) at (2, 0) {$\circ$};
    \node (A1_3) at (3, 0) {$\circ$};
    \path (A1_1) edge [-] node [auto] {$\scriptstyle{}$} (A1_2);
        \path (A1_2) edge [-] node [auto] {$\scriptstyle{}$} (A1_3);
  \end{tikzpicture}$$ with integers $a,b,c>0$. Then Lemma~\ref{lem:duality} implies that $P$ contains a subgraph of the form $$\begin{tikzpicture}[xscale=1.0,yscale=1,baseline={(0,0)}]
    \node at (-1-.1, .4) {$-3$};
    \node at (0-.1, .4) {$-2$};
    \node at (2-.1, .4) {$-2$};
    \node at (3-.1, .4) {$-3$};
        \node (A1_9) at (-1, 0) {$\bullet$};
    \node (A1_0) at (0, 0) {$\bullet$};
    \node (A1_1) at (1, 0) {$\cdots$};
    \node (A1_2) at (2, 0) {$\bullet$};
    \node (A1_3) at (3, 0) {$\bullet$};
           
        \path (A1_9) edge [-] node [auto] {$\scriptstyle{}$} (A1_0);
    \path (A1_0) edge [-] node [auto] {$\scriptstyle{}$} (A1_1);
    \path (A1_1) edge [-] node [auto] {$\scriptstyle{}$} (A1_2);
        \path (A1_2) edge [-] node [auto] {$\scriptstyle{}$} (A1_3);
  \end{tikzpicture},$$ where there are $b-1$ vertices of weight $-2$. Irrespective of the value of $b$ this shows that $P$ contains one the sub-plumbing graphs \ref{it:2-2}, \ref{it:33}, \ref{it:323}, or \ref{it:3223} from Theorem~\ref{thm:main}\ref{it:combinatorial}. This establishes Condition~\ref{it:3adjacent}.

Condition~\ref{it:weight3condition} is established by noting that if $P^*$ contains a vertex with adjusted weight three and a sub-plumbing graph of the form $$\begin{tikzpicture}[xscale=1.0,yscale=1,baseline={(0,0)}]
    \node at (1.0, .4) {$1$};
    \node at (2.0, .4) {$1$};
    \node (A1_1) at (1, 0) {$\circ$};
    \node (A1_2) at (2, 0) {$\circ$};
    \path (A1_1) edge [-] node [auto] {$\scriptstyle{}$} (A1_2);
  \end{tikzpicture},$$ then this would give rise to a configuration of the form \ref{it:3-22} from Theorem~\ref{thm:main}\ref{it:combinatorial} in $P$.
  
Now suppose that $P^*$ contains a chain of adjacent vertices $v_0, \dots, v_{k+1}$ where 
$$\text{$w'(v_0)\geq 1$, $w'(v_1)=\dots = w'(v_{k})=0$, $w'(v_{k+1})=1$ and $k \neq 0$.}$$
Condition~\ref{it:largeweight} implies that $w'(v_0)\in\{1,2,3\}$. We consider each of these possibilities in turn.
\begin{itemize}
\item If $w'(v_0)=1$, then this corresponds to a subgraph of the form
\[\begin{tikzpicture}[xscale=1.0,yscale=1,baseline={(0,0)}]
    \node at (1-.1, .4) {$-(k+3)$};
    \node (A1_1) at (1, 0) {$\bullet$};
  \end{tikzpicture}\]
in $P$. Thus in order to avoid a copy of Theorem~\ref{thm:main}\ref{it:combinatorial} \ref{it:4} in $P$ we must have $k\geq 2$.
\item If $w'(v_0)=2$, then this corresponds to a subgraph of the form
\[\begin{tikzpicture}[xscale=1.0,yscale=1,baseline={(0,0)}]
    \node at (1-.1, .4) {$-(k+3)$};
    \node at (2-.1, .4) {$-2$};
    \node (A1_1) at (1, 0) {$\bullet$};
    \node (A1_2) at (2, 0) {$\bullet$};
    \path (A1_1) edge [-] node [auto] {$\scriptstyle{}$} (A1_2);
  \end{tikzpicture}\]
in $P$. Thus in order to avoid a copy of Theorem~\ref{thm:main}\ref{it:combinatorial} \ref{it:4} and \ref{it:52} in $P$ we must have $k\geq 3$.
\item If $w'(v_0)=3$, then this corresponds to a subgraph of the form
\[\begin{tikzpicture}[xscale=1.0,yscale=1,baseline={(0,0)}]
    \node at (1-.1, .4) {$-(k+3)$};
    \node at (2-.1, .4) {$-2$};
    \node at (3-.1, .4) {$-2$};
    \node (A_1) at (1, 0) {$\bullet$};
    \node (A_2) at (2, 0) {$\bullet$};
    \node (A_3) at (3, 0) {$\bullet$};
        \path (A_1) edge [-] node [auto] {$\scriptstyle{}$} (A_2);
    \path (A_2) edge [-] node [auto] {$\scriptstyle{}$} (A_3);
  \end{tikzpicture}\]
in $P$. Thus in order to avoid a copy of Theorem~\ref{thm:main}\ref{it:combinatorial} \ref{it:4}, \ref{it:52} and \ref{it:622} in $P$ we must have $k\geq 4$.
\end{itemize}
In all cases, we have established $k\geq w'(v_0)+1$ which gives Condition~\ref{it:long_chain}.

Finally, we check that if $P^*$ contains one of the configurations of Condition~\ref{it:forbidden_configs}, then $P$ would contain a subgraph containing one configurations listed in Theorem~\ref{thm:main}\ref{it:combinatorial}. The correspondence is as follows: 
\begin{enumerate}[label=(\alph*),font=\upshape]
\item $\begin{tikzpicture}[xscale=1.0,yscale=1,baseline={(0,0)}]
    \node at (-2.0+0.0, .4) {$1$};
    \node at (-1.0+0.0, .4) {$1$};
    \node at (0.0, .4) {$0$};
    \node at (1.0, .4) {$0$};
    \node at (2.0, .4) {$1$};
    \node at (3.0, .4) {$2$};
        \node (A1_8) at (-2, 0) {$\circ$};
        \node (A1_9) at (-1, 0) {$\circ$};
    \node (A1_0) at (0, 0) {$\circ$};
    \node (A1_1) at (1, 0) {$\circ$};
    \node (A1_2) at (2, 0) {$\circ$};
    \node (A1_3) at (3, 0) {$\circ$};
            \path (A1_8) edge [-] node [auto] {$\scriptstyle{}$} (A1_9);
        \path (A1_9) edge [-] node [auto] {$\scriptstyle{}$} (A1_0);
    \path (A1_0) edge [-] node [auto] {$\scriptstyle{}$} (A1_1);
    \path (A1_1) edge [-] node [auto] {$\scriptstyle{}$} (A1_2);
        \path (A1_2) edge [-] node [auto] {$\scriptstyle{}$} (A1_3);
  \end{tikzpicture}$ gives rise to a subgraph of the form $$\begin{tikzpicture}[xscale=1.0,yscale=1,baseline={(0,0)}]
    \node at (1-.1, .4) {$-3$};
    \node at (2-.1, .4) {$-5$};
    \node at (3-.1, .4) {$-3$};
    \node at (4-.1, .4) {$-2$};
    \node (A1_1) at (1, 0) {$\bullet$};
    \node (A1_2) at (2, 0) {$\bullet$};
    \node (A1_3) at (3, 0) {$\bullet$};
    \node (A1_4) at (4, 0) {$\bullet$};
    \path (A1_2) edge [-] node [auto] {$\scriptstyle{}$} (A1_3);
    \path (A1_3) edge [-] node [auto] {$\scriptstyle{}$} (A1_4);
    \path (A1_1) edge [-] node [auto] {$\scriptstyle{}$} (A1_2);
  \end{tikzpicture}$$ and

\item $\begin{tikzpicture}[xscale=1.0,yscale=1,baseline={(0,0)}]
    \node at (-1.0+0.0, .4) {$3$};
    \node at (0.0, .4) {$1$};
    \node at (1.0, .4) {$0$};
    \node at (2.0, .4) {$0$};
    \node at (3.0, .4) {$1$};
        \node (A1_9) at (-1, 0) {$\circ$};
    \node (A1_0) at (0, 0) {$\circ$};
    \node (A1_1) at (1, 0) {$\circ$};
    \node (A1_2) at (2, 0) {$\circ$};
    \node (A1_3) at (3, 0) {$\circ$};
        \path (A1_9) edge [-] node [auto] {$\scriptstyle{}$} (A1_0);
    \path (A1_0) edge [-] node [auto] {$\scriptstyle{}$} (A1_1);
    \path (A1_1) edge [-] node [auto] {$\scriptstyle{}$} (A1_2);
        \path (A1_2) edge [-] node [auto] {$\scriptstyle{}$} (A1_3);
  \end{tikzpicture}$ gives rise to a subgraph of the form $$\begin{tikzpicture}[xscale=1.0,yscale=1,baseline={(0,0)}]
    \node at (1-.1, .4) {$-2$};
    \node at (2-.1, .4) {$-2$};
    \node at (3-.1, .4) {$-3$};
    \node at (4-.1, .4) {$-5$};
    \node (A1_1) at (1, 0) {$\bullet$};
    \node (A1_2) at (2, 0) {$\bullet$};
    \node (A1_3) at (3, 0) {$\bullet$};
    \node (A1_4) at (4, 0) {$\bullet$};
    \path (A1_2) edge [-] node [auto] {$\scriptstyle{}$} (A1_3);
    \path (A1_3) edge [-] node [auto] {$\scriptstyle{}$} (A1_4);
    \path (A1_1) edge [-] node [auto] {$\scriptstyle{}$} (A1_2);
  \end{tikzpicture}$$.
\end{enumerate}
This establishes Condition~\ref{it:forbidden_configs}.
\end{proof}

\section{Rigid linear lattices}\label{sec:lattice_analysis}
The objective of this section is to prove Theorem~\ref{thm:technical}. Throughout this section, we will take $\Gamma$ to be a disjoint union of linear plumbing graphs. The aim is to show that if $\Gamma$ satisfies the Working Conditions, then the corresponding lattice $Q_\Gamma$ is rigid. By Lemma~\ref{lem:forbidden_configurations} the dual of a plumbing graph satisfying Theorem~\ref{thm:main}\ref{it:combinatorial} satisfies the Working Conditions, therefore this is sufficient to establish Theorem~\ref{thm:technical}.

To this end, we introduce some terminology specific to the problem of mapping linear lattices into diagonal lattices. Throughout this section, for simplicity we refer to the standard positive-definite lattice $(\Z^n,\langle 1 \rangle^n)$ as $\Z^n$.

\begin{defn} An {\em embedding} of $\Gamma$ into $\Z^n$ is a function 
$$\phi: V(\Gamma) \rightarrow \Z^n$$
such that for all $u,v\in V(\Gamma)$, we have
\[
\phi(u)\cdot \phi(v)= 
\begin{cases}
w(u) &\text{if $u=v$}\\
-1 &\text{if $u$ and $v$ adjacent}\\
0 &\text{otherwise.}
\end{cases}
\]\end{defn}
Since lattice maps of $Q_\Gamma$ into $\Z^n$ are determined by the image of the vertices, embeddings of $\Gamma$ into $\Z^n$ are in bijection with lattice maps of $Q_\Gamma$ into $\Z^n$. Similarly, we define what it means for a plumbing graph to be rigid.

\begin{defn}
A plumbing graph $\Gamma$ is {\em rigid} if for any integer $n\geq 1$ and any pair of embeddings $\phi, \phi': V(\Gamma)\rightarrow \Z^n$, there is an element $\alpha \in \aut(\Z^n)$ such that $\phi'(v) = \alpha(\phi(v))$ for each $v \in V(\Gamma)$.
\end{defn}

Thus the lattice $Q_\Gamma$ is rigid in the sense of Definition~\ref{def:rigidity} if and only if $\Gamma$ is rigid as defined in terms of embeddings here.

For a vector $v\in \Z^n$, we define the \emph{support} of $v$ to be
\[
\supp(v):=\{e\in \Z^n \, | \, \norm{e}=1 \, \text{ and } \, e\cdot v \neq 0 \}/( e\sim -e)
\]
that is the set of unit vectors pairing nontrivially with $v$ up to sign. For vectors $v_1, \ldots, v_k \in \Z^n$, we define the support of $v_1, \ldots, v_k$ to be
$$\supp(v_1,\ldots, v_k):=\supp(v_1) \cup \cdots \cup\supp(v_k).$$

\vspace{-.2cm}
\begin{rem}We quotient out the support by the relation $e\sim -e$, so as to avoid explicitly choosing an orthonormal basis for $\Z^n$. However, we see that for any choice of orthonormal basis $\supp(v)$ contains one element for each nonzero coordinate of $v$. This gives the inequality
\begin{equation}
\norm{v}\geq |\supp(v)|
\end{equation}
with equality if and only if $v=e_1+\dots + e_k$ for some collection of distinct orthogonal unit vectors $e_1, \dots, e_k\in \Z^n$.
\end{rem}

\begin{defn}\label{def:standard}
An embedding $\phi: V(\Gamma) \rightarrow \Z^n$ is \emph{standard} if 
\[
|\supp(\phi(u))\cap \supp(\phi(v))| = \begin{cases}
w(v) &\text{if $u=v$}\\
1 & \text{if $u$ and $v$ are adjacent}\\
0&\text{otherwise}
\end{cases}
\] 
for all vertices $u$ and $v$ in $V(\Gamma)$.
\end{defn}

\begin{rem}\label{rem:standard_embeddings} We make the following observations regarding standard embeddings:
\begin{enumerate}[label=(\alph*),font=\upshape]
\item\label{it:standard_bound} The plumbing graph $\Gamma$ always admits standard embeddings in $\Z^n$ provided $n$ is sufficiently large. Explicitly one can calculate that a standard embedding requires $-E + \sum_{v\in V(\Gamma)} w(v)$ distinct orthogonal unit vectors, where $E$ is the number of edges in $\Gamma$. Thus a standard embedding exists if and only if $$n\geq -E + \sum_{v\in V(\Gamma)} w(v).$$
\item If $\phi$ is a standard embedding of $\Gamma$ and $\alpha \in \aut(\Z^n)$ is an automorphism, then $\alpha\circ \phi$ is also standard. Conversely any pair of standard embeddings in $\Z^n$ are related by an automorphism. Thus a plumbing graph $\Gamma$ is rigid if and only if every embedding of $\Gamma$ is standard.
\end{enumerate}

\end{rem}

Next we define a relative version of rigidity.
\begin{defn}
We say that a subgraph $\Gamma'$ of a plumbing graph $\Gamma$ is {\em rigid} if for any integer $n\geq 1$ and for any pair of embeddings $\phi, \phi': V(\Gamma)\rightarrow \Z^n$, there is an element $\alpha \in \aut(\Z^n)$ such that $\phi'(v) = \alpha(\phi(v))$ for each $v \in V(\Gamma')$.\end{defn}

For the rest of the article, whenever we are considering a plumbing graph $\Gamma$ together with a lattice embedding $\phi: V(\Gamma) \rightarrow \Z^n$ we will, by abuse of notation,  denote $\phi(v)$ by $v$ for each $v\in V(\Gamma)$.

\begin{lem}\label{lem:2chains_rigid}
If $\Gamma$ is a disjoint union of linear plumbing graphs satisfying the Working Conditions, then the subgraph $\Gamma'$ containing all vertices of weight two is rigid.
\end{lem}
\begin{proof}
Let $\phi : V(\Gamma) \rightarrow \Z^n$ be an embedding. First, note that for any vertex $v$ of weight two we have $|\supp(v)|=2$, which is to say $v$ can be written in the form $v=e_1+e_2$ for two distinct orthogonal unit vectors in $\Z^n$. Secondly, note that if $v_1$ and $v_2$ are two adjacent vertices of weight two, we necessarily have that  $|\supp(v_1)\cap\supp(v_2)|=1$.  Thus it remains to show that if $v_1$ and $v_2$ are two nonadjacent vertices of weight two, then $\supp(v_1)$ and $\supp(v_2)$ must be disjoint. If these supports were not disjoint, then we would have that $\supp(v_1)=\supp(v_2)$ and up to automorphisms of $\Z^n$ we could write $v_1$ in the form $v_1= e_1-e_2$ and $v_2=e_1 +e_2$. However, in this case we have that $v \cdot v_1 \equiv v \cdot v_2 \bmod{2}$ for any vertex $v$, so every vertex which is adjacent to $v_1$ must also be adjacent to $v_2$. Since $\Gamma$ does not contain any cycles and $\Gamma$ contains at most one vertex with adjusted weight $w'(v)>1$, we conclude that $v_1$ and $v_2$ lie in a connected component of $\Gamma$ taking the form $$\begin{tikzpicture}[xscale=1.0,yscale=1,baseline={(0,0)}]
    \node at (1, .4) {$2$};
    \node at (2, .4) {$n$};
    \node at (3, .4) {$2$};
    \node (A1_1) at (1, 0) {$\bullet$};
    \node (A1_2) at (2, 0) {$\bullet$};
    \node (A1_3) at (3, 0) {$\bullet$};
    \path (A1_1) edge [-] node [auto] {$\scriptstyle{}$} (A1_2);
        \path (A1_2) edge [-] node [auto] {$\scriptstyle{}$} (A1_3);
  \end{tikzpicture}$$ for some $n\geq 2$. Such a component has adjusted weights $$\begin{tikzpicture}[xscale=1.0,yscale=1,baseline={(0,0)}]
    \node at (1.0, .4) {$1$};
    \node at (2.0, .4) {$n-2$};
    \node at (3.0, .4) {$1$};
    \node (A1_1) at (1, 0) {$\circ$};
    \node (A1_2) at (2, 0) {$\circ$};
    \node (A1_3) at (3, 0) {$\circ$};
    \path (A1_1) edge [-] node [auto] {$\scriptstyle{}$} (A1_2);
        \path (A1_2) edge [-] node [auto] {$\scriptstyle{}$} (A1_3);
  \end{tikzpicture}$$ 
 However, such a configuration is forbidden by the Working Conditions: Condition~\ref{it:3adjacent} rules out the case $n>2$ and Condition~\ref{it:long_chain} rules out the case $n=2$. Thus we have that $\supp(v_1)$ and $\supp(v_2)$ are disjoint, as required. 
\end{proof}

\begin{lem}\label{lem:0-chain_rigid}
If $\Gamma$ is a disjoint union of linear plumbing graphs satisfying the Working Conditions, then any subgraph of $\Gamma$ with adjusted weights

$$\begin{tikzpicture}[xscale=1.0,yscale=1,baseline={(0,0)}]
    \node at (1.0, .4) {$1$};
    \node at (2.0, .4) {$0$};
    \node at (4.0, .4) {$0$};
	\node at (5.0, .4) {$1$};
    \node (A1_1) at (1, 0) {$\circ$};
    \node (A1_2) at (2, 0) {$\circ$};
    \node (A1_3) at (3, 0) {$\cdots$};
        \node (A1_4) at (4, 0) {$\circ$};
            \node (A1_5) at (5, 0) {$\circ$};
    \path (A1_1) edge [-] node [auto] {$\scriptstyle{}$} (A1_2);
        \path (A1_2) edge [-] node [auto] {$\scriptstyle{}$} (A1_3);
                \path (A1_3) edge [-] node [auto] {$\scriptstyle{}$} (A1_4);
                        \path (A1_4) edge [-] node [auto] {$\scriptstyle{}$} (A1_5);
  \end{tikzpicture}$$
and containing at least three vertices is rigid.
\end{lem}
\begin{proof}\setcounter{case}{0}
Suppose that we have a chain of vertices $v_1, \dots, v_k$ of $\Gamma$ with 
\[w'(v_1)=w'(v_k)=1 \qquad \text{ and } \qquad w'(v_2)=\dots = w'(v_{k-1})=0.\]
By hypothesis, we are assuming that $k\geq 3$. Condition~\ref{it:long_chain} implies that $k\geq 4$. We break down the proof into four cases.

\begin{case}
Both $v_1$ and $v_k$ are leaves in $\Gamma$.
\end{case}

If both $v_1$ and $v_k$ are leaves, then $v_1, \dots, v_k$ are the vertices of a connected component of $\Gamma$ of the form
$$
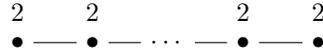
\begin{tikzpicture}[xscale=1.0,yscale=1,baseline={(0,0)}]
    \node at (1, .4) {$2$};
    \node at (2, .4) {$2$};
    \node at (4, .4) {$2$};
	\node at (5, .4) {$2$};
    \node (A1_1) at (1, 0) {$\bullet$};
    \node (A1_2) at (2, 0) {$\bullet$};
    \node (A1_3) at (3, 0) {$\cdots$};
        \node (A1_4) at (4, 0) {$\bullet$};
            \node (A1_5) at (5, 0) {$\bullet$};
    \path (A1_1) edge [-] node [auto] {$\scriptstyle{}$} (A1_2);
        \path (A1_2) edge [-] node [auto] {$\scriptstyle{}$} (A1_3);
                \path (A1_3) edge [-] node [auto] {$\scriptstyle{}$} (A1_4);
                        \path (A1_4) edge [-] node [auto] {$\scriptstyle{}$} (A1_5);
  \end{tikzpicture}$$
This is rigid by Lemma~\ref{lem:2chains_rigid}.

\begin{case}
Exactly one of $v_1$ and $v_k$ is a leaf in $\Gamma$.
\end{case}
Without loss of generality, we may assume that $v_1$ is the leaf. Thus $v_k$ has degree 2 and hence $w(v_k)=3$. By Lemma~\ref{lem:2chains_rigid}, we can assume that for $i=1, \dots, k-1$ we have $v_i=e_{i}-e_{i-1}$, where $e_0, \dots, e_{k-1}$ are distinct orthogonal unit vectors in $\Z^n$. Thus in order for $v_k$ to have the correct pairings with the other $v_i$ we see that its embedding must take the form
\[
v_k=a(e_0 + \dots +e_{k-1})-e_{k-1}+ v',
\]
where $a$ is an integer and $v'$ is some vector satisfying
$v'\cdot e_0= \dots = v'\cdot e_{k-1}=0$. If $a=0$, then this is a standard embedding since $v_k$ takes the form $v_k=-e_{k-1}+v'$ and so we $|\supp(v_k)\cap \supp(v_{k-1})|=1$ and $\supp(v_k)\cap \supp(v_{i})=\varnothing$ for $i=1, \dots, k-2$.
Thus we can assume that $a\neq 0$. Computing the norm of $v_k$ yields that
\[
w(v_k)=3=(k-1) a^2+(a-1)^2+\norm{v'}\geq (k-1)+(a-1)^2 +\norm{v'}.
\]
Since $k\geq 4$, this is possible only if $k=4$, $a=1$ and $v'=0$. That is, $v_4=v_k$ is embedded in the form $v_4=e_0+e_1+e_2$. Since $v_4$ is not a leaf, there exists a vertex $v_5$ satisfying
\[\text{$v_5\cdot v_4=-1$ and $v_5\cdot v_1 =v_5\cdot v_2= v_5\cdot v_3=0$}.\]
The pairings of $v_5$ with $v_1, v_2$ imply that $v_5\cdot e_0=v_5\cdot e_1=v_5\cdot e_2$. Thus we have that
\[
v_4\cdot v_5=v_5\cdot(e_0+ e_1+e_2)=3v_5\cdot e_0 \not\equiv -1 \bmod 3.
\]
This is a contradiction showing that no such $v_5$ exists and, consequently, ruling out the possibility that $a\neq 0$.
\begin{case}
Neither $v_1$ nor $v_k$ is a leaf and $k\geq 5$.
\end{case}
In this case, $w(v_1)=w(v_k)=3$. By Lemma~\ref{lem:2chains_rigid}, we can assume that we have $v_i=e_{i}-e_{i-1}$ for $i=2, \dots, k-1$, where $e_1, \dots, e_{k-1}$ are distinct orthogonal unit vectors in $\Z^n$. In order to have the correct pairings, we see that $v_1$ and $v_k$ must take the forms 
\[\text{$v_1=a(e_1 + \dots +e_{k-1})+e_{1} + v'$ and $v_k=b(e_1 + \dots +e_{k-1})-e_{k-1} + v''$,}\]
where $a$ and $b$ are integers and $v'$ and $v''$ are vectors satisfying 
\[\text{$v'\cdot e_1= \dots = v'\cdot e_{k-1}=0$ and $v''\cdot e_1= \dots = v''\cdot e_{k-1}=0$.}\]

First we establish that $a=b=0$. Suppose that one of $a$ or $b$ is nonzero. Without loss of generality, suppose that $a\neq 0$. Computing the norm of $v_1$ gives
\[
3=w(v_1) =(a+1)^2 + (k-2) a^2 + \norm{v'}\geq (a+1)^2 + (k-2) a^2.
\]
Since $k\geq 5$, the assumption that $a\neq 0$ implies that $v'=0$, $a=-1$ and $k=5$. This puts $v_1$ in the form $v_1=-e_2-e_3-e_4$. However for such a $v_1$, we have $v_1\cdot v_k =-3b+1$. This contradicts the fact that $v_1\cdot v_k=0$. Thus we are forced to conclude that $a=b=0$.

Thus we have $v_1$ in the form $v_1=e_1+v'$ and $v_k$ in the form $v_k=-e_{k-1}+v''$. Thus to show rigidity it remains to verify that $\supp(v_1)\cap \supp(v_k)=\varnothing$.

In particular, it suffices to assume that $\supp(v')\cap \supp(v'')\neq\varnothing$. Both $v'$ and $v''$ are vectors of norm two and they must be orthogonal to each other in order to ensure that $v_1\cdot v_k=0$. Thus we can assume for a contradiction that $v'=e_{k}+e_{k+1}$ and $v''=e_{k}-e_{k+1}$. Hence $v_1$ and $v_k$ take the forms $v_1=e_1+ e_{k}+e_{k+1}$ and $v_k=-e_{k-1}+e_{k}-e_{k+1}$. Now let $u$ be any other vertex. This has zero pairing with $v_2, \dots v_{k-1}$ so it must take the form
\[
u=c(e_1+ \dots + e_{k-1})+u'
\]
for some integer $c$ and $u'$ satisfying $u'\cdot e_1= \dots = u'\cdot e_{k-1}=0$. However, for such $u$ we have $u\cdot v_1 \equiv u\cdot v_k \bmod{2}$. This shows that there is no vertex adjacent to $v_1$ but not to $v_k$. This is a contradiction allowing us to conclude that the subgraph given by $v_1,\dots, v_k$ is rigid.

\begin{case}
Neither $v_1$ nor $v_k$ are leaves in $\Gamma$ and $k=4$.
\end{case}
This is the final and most elaborate case and we reduce it to a computer calculation. We will show that $v_1,v_2,v_3, v_4$ arises in one of the following configurations, where the weights of $v_1, v_2,v_3,v_4$ are circled:
\begin{enumerate}
\item\label{item:computer1}
$\begin{tikzpicture}[xscale=1.0,yscale=1,baseline={(0,0)}]
    \node at (1, .4) {$2$};
    \node at (2, .4) {$\red{\circled{3}}$};
    \node at (3, .4) {$\red{\circled{2}}$};
    \node at (4, .4) {$\red{\circled{2}}$};
        \node at (5, .4) {$\red{\circled{3}}$};
                \node at (6, .4) {$2$};
    \node (A1_1) at (1, 0) {$\bullet$};
    \node (A1_2) at (2, 0) {$\bullet$};
    \node (A1_3) at (3, 0) {$\bullet$};
    \node (A1_4) at (4, 0) {$\bullet$};
        \node (A1_5) at (5, 0) {$\bullet$};
            \node (A1_6) at (6, 0) {$\bullet$};
    \path (A1_2) edge [-] node [auto] {$\scriptstyle{}$} (A1_3);
    \path (A1_3) edge [-] node [auto] {$\scriptstyle{}$} (A1_4);
    \path (A1_1) edge [-] node [auto] {$\scriptstyle{}$} (A1_2);
        \path (A1_4) edge [-] node [auto] {$\scriptstyle{}$} (A1_5);
            \path (A1_5) edge [-] node [auto] {$\scriptstyle{}$} (A1_6);
  \end{tikzpicture}$

\item\label{item:computer2}
$\begin{tikzpicture}[xscale=1.0,yscale=1,baseline={(0,0)}]
    \node at (1, .4) {$2$};
    \node at (2, .4) {$\red{\circled{3}}$};
    \node at (3, .4) {$\red{\circled{2}}$};
    \node at (4, .4) {$\red{\circled{2}}$};
        \node at (5, .4) {$\red{\circled{3}}$};
\node at (6, .4) {$3$};
\node at (7, .4) {$2$};
                                
    \node (A1_1) at (1, 0) {$\bullet$};
    \node (A1_2) at (2, 0) {$\bullet$};
    \node (A1_3) at (3, 0) {$\bullet$};
    \node (A1_4) at (4, 0) {$\bullet$};
        \node (A1_5) at (5, 0) {$\bullet$};
            \node (A1_6) at (6, 0) {$\bullet$};
\node (A1_7) at (7, 0) {$\bullet$};

    \path (A1_2) edge [-] node [auto] {$\scriptstyle{}$} (A1_3);
    \path (A1_3) edge [-] node [auto] {$\scriptstyle{}$} (A1_4);
    \path (A1_1) edge [-] node [auto] {$\scriptstyle{}$} (A1_2);
        \path (A1_4) edge [-] node [auto] {$\scriptstyle{}$} (A1_5);
\path (A1_5) edge [-] node [auto] {$\scriptstyle{}$} (A1_6);
\path (A1_6) edge [-] node [auto] {$\scriptstyle{}$} (A1_7);
  \end{tikzpicture}$
  
\item\label{item:computer3} $\begin{tikzpicture}[xscale=1.0,yscale=1,baseline={(0,0)}]
    \node at (1, .4) {$2$};
    \node at (2, .4) {$2$};
    \node at (3, .4) {$\red{\circled{3}}$};
    \node at (4, .4) {$\red{\circled{2}}$};
        \node at (5, .4) {$\red{\circled{2}}$};
\node at (6, .4) {$\red{\circled{3}}$};
\node at (7, .4) {$3$};
                                
    \node (A1_1) at (1, 0) {$\bullet$};
    \node (A1_2) at (2, 0) {$\bullet$};
    \node (A1_3) at (3, 0) {$\bullet$};
    \node (A1_4) at (4, 0) {$\bullet$};
        \node (A1_5) at (5, 0) {$\bullet$};
            \node (A1_6) at (6, 0) {$\bullet$};
\node (A1_7) at (7, 0) {$\bullet$};

    \path (A1_2) edge [-] node [auto] {$\scriptstyle{}$} (A1_3);
    \path (A1_3) edge [-] node [auto] {$\scriptstyle{}$} (A1_4);
    \path (A1_1) edge [-] node [auto] {$\scriptstyle{}$} (A1_2);
        \path (A1_4) edge [-] node [auto] {$\scriptstyle{}$} (A1_5);
\path (A1_5) edge [-] node [auto] {$\scriptstyle{}$} (A1_6);
\path (A1_6) edge [-] node [auto] {$\scriptstyle{}$} (A1_7);
  \end{tikzpicture}$

\item\label{item:computer4} $\begin{tikzpicture}[xscale=1.0,yscale=1,baseline={(0,0)}]
    \node at (1, .4) {$2$};
    \node at (2, .4) {$2$};
    \node at (3, .4) {$\red{\circled{3}}$};
    \node at (4, .4) {$\red{\circled{2}}$};
        \node at (5, .4) {$\red{\circled{2}}$};
\node at (6, .4) {$\red{\circled{3}}$};
\node at (7, .4) {$4$};
\node at (8, .4) {$2$};
                                
    \node (A1_1) at (1, 0) {$\bullet$};
    \node (A1_2) at (2, 0) {$\bullet$};
    \node (A1_3) at (3, 0) {$\bullet$};
    \node (A1_4) at (4, 0) {$\bullet$};
        \node (A1_5) at (5, 0) {$\bullet$};
            \node (A1_6) at (6, 0) {$\bullet$};
\node (A1_7) at (7, 0) {$\bullet$};
\node (A1_8) at (8, 0) {$\bullet$};

    \path (A1_2) edge [-] node [auto] {$\scriptstyle{}$} (A1_3);
    \path (A1_3) edge [-] node [auto] {$\scriptstyle{}$} (A1_4);
    \path (A1_1) edge [-] node [auto] {$\scriptstyle{}$} (A1_2);
        \path (A1_4) edge [-] node [auto] {$\scriptstyle{}$} (A1_5);
\path (A1_5) edge [-] node [auto] {$\scriptstyle{}$} (A1_6);
\path (A1_6) edge [-] node [auto] {$\scriptstyle{}$} (A1_7);
\path (A1_7) edge [-] node [auto] {$\scriptstyle{}$} (A1_8);
  \end{tikzpicture}$

\item\label{item:computer5}$\begin{tikzpicture}[xscale=1.0,yscale=1,baseline={(0,0)}]
    \node at (1, .4) {$2$};
    \node at (2, .4) {$3$};
    \node at (3, .4) {$\red{\circled{3}}$};
    \node at (4, .4) {$\red{\circled{2}}$};
        \node at (5, .4) {$\red{\circled{2}}$};
\node at (6, .4) {$\red{\circled{3}}$};
\node at (7, .4) {$3$};
\node at (8, .4) {$2$};
                                
    \node (A1_1) at (1, 0) {$\bullet$};
    \node (A1_2) at (2, 0) {$\bullet$};
    \node (A1_3) at (3, 0) {$\bullet$};
    \node (A1_4) at (4, 0) {$\bullet$};
        \node (A1_5) at (5, 0) {$\bullet$};
            \node (A1_6) at (6, 0) {$\bullet$};
\node (A1_7) at (7, 0) {$\bullet$};
\node (A1_8) at (8, 0) {$\bullet$};

    \path (A1_2) edge [-] node [auto] {$\scriptstyle{}$} (A1_3);
    \path (A1_3) edge [-] node [auto] {$\scriptstyle{}$} (A1_4);
    \path (A1_1) edge [-] node [auto] {$\scriptstyle{}$} (A1_2);
        \path (A1_4) edge [-] node [auto] {$\scriptstyle{}$} (A1_5);
\path (A1_5) edge [-] node [auto] {$\scriptstyle{}$} (A1_6);
\path (A1_6) edge [-] node [auto] {$\scriptstyle{}$} (A1_7);
\path (A1_7) edge [-] node [auto] {$\scriptstyle{}$} (A1_8);
  \end{tikzpicture}$
\end{enumerate}
Since $v_1$ and $v_4$ are both of degree two we see that the plumbing graph locally takes the form
\[
\begin{tikzpicture}[xscale=1.0,yscale=1,baseline={(0,0)}]
    \node at (1, .4) {$a$};
    \node at (2, .4) {$\red{\circled{3}}$};
    \node at (3, .4) {$\red{\circled{2}}$};
    \node at (4, .4) {$\red{\circled{2}}$};
        \node at (5, .4) {$\red{\circled{3}}$};
                \node at (6, .4) {$b$};
    \node (A1_1) at (1, 0) {$\bullet$};
    \node (A1_2) at (2, 0) {$\bullet$};
    \node (A1_3) at (3, 0) {$\bullet$};
    \node (A1_4) at (4, 0) {$\bullet$};
        \node (A1_5) at (5, 0) {$\bullet$};
            \node (A1_6) at (6, 0) {$\bullet$};
    \path (A1_2) edge [-] node [auto] {$\scriptstyle{}$} (A1_3);
    \path (A1_3) edge [-] node [auto] {$\scriptstyle{}$} (A1_4);
    \path (A1_1) edge [-] node [auto] {$\scriptstyle{}$} (A1_2);
        \path (A1_4) edge [-] node [auto] {$\scriptstyle{}$} (A1_5);
            \path (A1_5) edge [-] node [auto] {$\scriptstyle{}$} (A1_6);
  \end{tikzpicture}
\] 
for some integers $a$ and $b$ satisfying $a, b \geq 2$. In terms of adjusted weights this is
$$\begin{tikzpicture}[xscale=1.0,yscale=1,baseline={(0,0)}]
    \node at (-2.0+0.04, .4) {$a'$};
    \node at (-1.0+0.0, .4) {\red{$\circled{1}$}};
    \node at (0.0, .4) {\red{$\circled{0}$}};
    \node at (1.0, .4) {\red{$\circled{0}$}};
    \node at (2.0, .4) {\red{$\circled{1}$}};
    \node at (3.0+0.04, .4) {$b'$};
        \node (A1_8) at (-2, 0) {$\circ$};
        \node (A1_9) at (-1, 0) {$\circ$};
    \node (A1_0) at (0, 0) {$\circ$};
    \node (A1_1) at (1, 0) {$\circ$};
    \node (A1_2) at (2, 0) {$\circ$};
    \node (A1_3) at (3, 0) {$\circ$};
            \path (A1_8) edge [-] node [auto] {$\scriptstyle{}$} (A1_9);
        \path (A1_9) edge [-] node [auto] {$\scriptstyle{}$} (A1_0);
    \path (A1_0) edge [-] node [auto] {$\scriptstyle{}$} (A1_1);
    \path (A1_1) edge [-] node [auto] {$\scriptstyle{}$} (A1_2);
        \path (A1_2) edge [-] node [auto] {$\scriptstyle{}$} (A1_3);
  \end{tikzpicture}$$
By symmetry, we can assume that $a'\leq b'$. Condition~\ref{it:largeweight} implies that $b'\leq 3$. However the forbidden configuration  from Condition~\ref{it:forbidden_configs} \ref{it:31001} rules out the possibility that $b'=3$. Thus we see that $b'\leq 2$ and furthermore that $a' \leq 1$ if $b'=2$ by Condition~\ref{it:largeweight}. This gives us the following possibilities for $a'$ and $b'$:  $(a',b')\in \{(0,0), (0,1), (0,2), (1,1),(1,2)\}$. We consider each in turn:
\begin{itemize}
\item If $a'=b'=0$, then we necessarily have $a=b=2$, which is case~\eqref{item:computer1}.
\item If $a'=0$ and $b'=1$, then we have that $a=2$ and $b\in\{2,3\}$. If $a=b=2$, then we are in case~\eqref{item:computer1} again. If $a=2$ and $b=3$, then we are in case~\eqref{item:computer2} since the vertex of weight $b=3$ and adjusted weight $b'=1$ must be adjacent to a vertex of adjusted weight zero by Condition~\ref{it:3adjacent}.

\item If $a'=0$ and $b'=2$, then we have that $a=2$ and $b\in\{3,4\}$. In either case, by Condition~\ref{it:long_chain} the vertex of weight $a=2$ is adjacent to a vertex of adjusted weight zero and hence of weight two. If $b=3$, then we must be in case~\eqref{item:computer3}. If $b=4$, then we must be in case~\eqref{item:computer4}, since the vertex of weight $b=4$ and adjusted weight $b'=2$ must be adjacent to a vertex of adjusted weight zero by Condition~\ref{it:3adjacent}.

\item If $a'=b'=1$, then we can assume by symmetry that $a\leq b$. If $a=b=2$, then we have case~\eqref{item:computer1}. If $a=2$ and $b=3$, then we have case~\eqref{item:computer2}, since the vertex with weight $b=3$ must be adjacent to a vertex of adjusted weight zero by Condition~\ref{it:3adjacent}.  If $a=3$ and $b=3$, then we have case~\eqref{item:computer5} since both vertices of weights $a$ and $b$ must be adjacent to vertices with adjusted weight zero by Condition~\ref{it:3adjacent}.
\item The possibility that $a'=1$ and $b'=2$ cannot occur, since this would yield the forbidden configuration \ref{it:110012} from Condition~\ref{it:forbidden_configs}.
\end{itemize}

In order to complete the proof of the lemma, we used the GAP computer algebra system \cite{GAP} to enumerate all possible embeddings of the five configurations \eqref{item:computer1}--\eqref{item:computer5}. One finds that there are 14 embeddings to consider and in each of these embeddings the subgraph given by $v_1, v_2, v_3, v_4$ has a standard embedding. In particular, this implies that the subgraph given by $v_1, v_2, v_3, v_4$ is rigid, since it implies that no nonstandard embedding of these vertices can extend to an embedding of $\Gamma$. For further details see Appendix~\ref{sec:computer} and Lemma~\ref{lem:computer1}, in particular.
\end{proof}

\begin{defn}\label{def:chain}
A \emph{chain of twos} in a plumbing graph $\Gamma$ is a maximal connected subgraph $v_1, \dots, v_k$ such that each vertex $v_i$ has weight two and $k\geq 2$. Here maximality means that it is not contained in any larger connected subgraph consisting entirely of vertices of weight two.
\end{defn}

The following lemma shows that there are typically many chains of twos in a plumbing graph satisfying the Working Conditions.
\begin{lem}\label{lem:degreetwochains}
Let $\Gamma$ be a disjoint union of linear plumbing graphs satisfying the Working Conditions and let $v$ be a vertex of $\Gamma$ with degree $d(v)=2$. Then either $v$ is contained in a chain of twos or $v$ is adjacent to a chain of twos.
\end{lem}
\begin{proof}
First observe that for any vertex $u$ with adjusted weight $w'(u)=0$, Condition~\ref{it:long_chain} implies that $u$ is contained in a chain of twos. Thus if $w'(v)=0$, then we conclude immediately that $v$ is contained in a chain of twos. On the other hand, if $w'(v)\geq 1$, then Condition~\ref{it:3adjacent} implies that $v$ is adjacent to a vertex $u$ with adjusted weight $w'(u)=0$. This vertex $u$ is contained in a chain of twos as before.
\end{proof}

The following lemma will play an essential role in the induction procedure.

\begin{lem}\label{lem:min_example_properties}
Let $\Gamma$ be a disjoint union of linear plumbing graphs satisfying the Working Conditions and containing a chain of twos $v_1, \dots , v_k$. If $\Gamma$ admits an embedding with a vertex $u$ such that 
$$|\supp(u) \cap \supp(v_1, \dots, v_k)|\geq 2,$$
then
\begin{enumerate}[label=(\alph*),font=\upshape]
\item\label{it:supp_bound} $|\supp(u) \cap \supp(v_1, \dots, v_k)|\geq 3$,
\item\label{it:u_large_norm} $u$ is the unique vertex in $\Gamma$ with $w'(u)>1$, and
\item\label{it:min_candidate} $\Gamma$ contains at least two chains of twos.

\end{enumerate}
\end{lem}
\begin{proof}
By Lemma~\ref{lem:2chains_rigid}, we can suppose that $v_1, \dots, v_k$ embed as $v_i=e_i-e_{i-1}$, where $e_0, \dots, e_{k}$ are distinct orthogonal unit vectors in $\Z^n$. We subdivide the analysis into several cases.
\setcounter{case}{0}
\begin{case}
Both $v_1$ and $v_k$ are leaves.
\end{case}
Since $v_1$ and $v_k$ are leaves, the vertices $v_1, \dots, v_k$ form a connected component of $\Gamma$ and $u$ is not adjacent to either $v_1$ or $v_k$. By Condition~\ref{it:long_chain} we have either $k=2$ or $k\geq  4$. In this case, the vertex $u$ must take the form
\[
u=a(e_0 + \dots + e_k)+u',
\]
where $a$ is a nonzero integer and $u'\cdot e_0= \dots = u'\cdot e_{k}=0$. Since $a$ is nonzero and $k\geq 2$, we evidently have \ref{it:supp_bound} of the lemma in this case. 

Moreover, computing the norm of $u$ shows that $w(u)\geq k+1$ with equality if and only if
$u=\pm (e_0 + \dots + e_k)$. However if $u=\pm (e_0 + \dots + e_k)$, then $u$ must be an isolated vertex, since we cannot find a vector $z\in \Z^n$ with $z\cdot u=-1$ and $z\cdot v_1=\dots = z\cdot v_k=0$. This implies that $w'(u)=w(u)=k+1\geq 3$ in this case. However, by Condition~\ref{it:largeweight} we have have $w'(u)\leq 3$, implying that $w'(u)=3$ and $k=2$. However, this violates Condition~\ref{it:weight3condition}, since if $k=2$, then $v_1$ and $v_2$ form two adjacent vertices with $w'(v_1)=w'(v_2)=1$. Thus we conclude that $w(u)\geq k+2\geq 4$. This implies that $w'(u)\geq 2$, giving \ref{it:u_large_norm} of the lemma.

Now we argue that $u$ must have degree two in $\Gamma$. If $d(u)=2$, then Lemma~\ref{lem:degreetwochains} and the fact that $w(u)>2$ implies that $u$ is adjacent to a chain of twos, implying the existence a chain of twos distinct from $v_1,\dots, v_k$ and verifying \ref{it:min_candidate} of the lemma. If $w(u)\geq 5$, then $d(u)=2$, since $w'(u)=w(u)-d(u)\leq 3$ by Condition~\ref{it:largeweight}. Thus suppose that $w(u)=k+2=4$. As we argued in the previous paragraph using Condition~\ref{it:weight3condition}, the fact that $k=2$ implies that there cannot be a vertex of adjusted weight three and hence $w'(u)=2$ giving $d(u)=2$ as required.

\begin{case}
Exactly one of $v_1$ or $v_k$ is a leaf.
\end{case}
Without loss of generality, suppose that $v_1$ is a leaf and $v_k$ is not a leaf. Thus $v_k$ is adjacent to some vertex $v_{k+1}$. By the maximality condition in Definition~\ref{def:chain}, we have $w(v_{k+1})\geq 3$. This implies that $w'(v_{k+1})\geq 1$. Furthermore, Condition~\ref{it:long_chain} implies that
\begin{equation}\label{eq:vk+1bound}
k-1\geq w'(v_{k+1})+1\geq 2.
\end{equation}
In order to have the correct pairings with $v_1, \dots, v_k$, we see that the vertex $v_{k+1}$ must take the form
\[
v_{k+1}= b(e_0+ \dots + e_k)-e_k +v',
\]
where  $b$ is an integer and $v'\cdot e_0= \dots = v'\cdot e_{k}=0$.

We will argue first that $b=0$, so that the hypothetical vertex $u$ that we are considering is not adjacent to $v_k$. To this end, suppose that $b$ is nonzero. Computing the norm of $v_{k+1}$ yields
$$w(v_{k+1}) =kb^2+(b-1)^2+ \norm{v'} \geq k$$
with equality if and only if $v_{k+1}=e_0+ \dots + e_{k-1}$. Combining this inequality with \eqref{eq:vk+1bound}, we have that
$$k\geq w'(v_{k+1})+2 \geq w(v_{k+1})\geq k.$$
Thus we have $w(v_{k+1})=w'(v_{k+1})+2$ and $v_{k+1}=e_0 +\dots + e_{k-1}$. Since $d(v_{k+1})=2$, there is another vertex adjacent to $v_{k+1}$. However, we cannot find a vertex $z$ that satisfies $z\cdot (e_0 +\dots + e_{k-1})=-1$ and $z\cdot v_i=0$ for $i=1, \dots, k$. It follows that $v_{k+1}$ must take the form $v_{k+1}= -e_k+v' $.

Thus we can assume $u$ must take the form
\[
u=a(e_0 + \dots + e_k)+u',
\]
where $u'\cdot e_0= \dots = u'\cdot e_{k}=0$ and $a$ is a nonzero integer. The fact that $k\geq 3$ establishes \ref{it:supp_bound} of the lemma.

Note that $w(u)\geq k+1$ with equality if and only if $u=e_0 + \dots + e_k$. \textit{A priori}, one also has $w(u)=k+1$ when $u$ takes the form $u=-(e_0 + \dots + e_k)$, but this cannot occur, being incompatible with the requirement that $u\cdot v_{k+1}\in\{0,-1\}$. Hence we have $w'(u) \geq w(u) -2 \geq k-1 \geq 2$, implying \ref{it:u_large_norm} of the lemma. 

Now, we wish to show that $\Gamma$ contains an additional chain of twos disjoint from $v_1,\dots, v_k$. If $d(u)=2$, then as in the previous case we certainly get an additional chain twos adjacent to $u$. Thus we can assume that $u$ has degree $d(u)\leq 1$.  In this case, Condition~\ref{it:largeweight} implies that
$$3 \geq w'(u) \geq w(u) -1 \geq k \geq 3.$$
Thus we have that $w(u)=k+1=4$, $d(u)=1$, and $u$ must take the form $u=e_0 + \dots + e_3$. Moreover, the pairing $u\cdot v_{k+1}=-1$ implies that $u$ must be adjacent to $v_{k+1}$. This implies that $\Gamma$ contains a connected component of the form
$$\begin{tikzpicture}[xscale=1.0,yscale=1,baseline={(0,0)}]
    \node at (-1.0, .4) {$2$};
    \node at (0.0, .4) {$2$};
    \node at (1.0, .4) {$2$};
    \node at (2.0, .4) {$3$};
    \node at (3.0, .4) {$4$};
        \node (A1_9) at (-1, 0) {$\bullet$};
    \node (A1_0) at (0, 0) {$\bullet$};
    \node (A1_1) at (1, 0) {$\bullet$};
    \node (A1_2) at (2, 0) {$\bullet$};
    \node (A1_3) at (3, 0) {$\bullet$};
        \path (A1_9) edge [-] node [auto] {$\scriptstyle{}$} (A1_0);
    \path (A1_0) edge [-] node [auto] {$\scriptstyle{}$} (A1_1);
    \path (A1_1) edge [-] node [auto] {$\scriptstyle{}$} (A1_2);
        \path (A1_2) edge [-] node [auto] {$\scriptstyle{}$} (A1_3);
  \end{tikzpicture}$$
which has adjusted weights forbidden by Condition~\ref{it:forbidden_configs} \ref{it:31001}. Thus we can rule out the possibility that $d(u)\leq 1$ and conclude \ref{it:min_candidate} of the lemma. 

\begin{case}
Neither $v_1$ nor $v_k$ is a leaf and $u$ is not adjacent to $v_1$ or $v_k$. 
\end{case}
Suppose that neither $v_1$ nor $v_k$ is a leaf. 
Assume the vertices $v_0$ and $v_{k+1}$ are distinct from $u$ and adjacent to $v_1$ and $v_{k}$, respectively. Since $v_1,\dots, v_k$ form a chain of twos, the maximality condition in Definition~\ref{def:chain} implies that $w(v_0),w(v_{k+1})\geq 3$. This in turn implies that $w'(v_0),w'(v_{k+1})\geq 1$. Note that Conditions~\ref{it:largeweight} and \ref{it:long_chain} imply that 
\begin{equation}\label{eq:kbound}
k\geq w'(v_0)+1 \qquad \text{ and }\qquad k\geq w'(v_{k+1})+1.
\end{equation}
In order to ensure that $u$ has zero pairing with $v_1, \dots, v_{k}$, we see that $u$ has to take the form
\[
u=a(e_0+ \dots + e_k)+u'
\]
where $a$ is a nonzero integer and $u'\cdot e_0 = \dots = u'\cdot e_k=0$. Since $k\geq 2$, this immediately establishes \ref{it:supp_bound} of the lemma. 

Computing the norm of $u$ yields $w(u)\geq k+1$ with equality if and only if $u$ takes the form $u=\pm(e_0+ \dots + e_k)$. If $w'(v_0)\geq 2$ or $w'(v_{k+1})\geq 2$, then \eqref{eq:kbound} implies $k\geq 3$ and hence 
\[w'(u) \geq w(u)-2\geq k-1 \geq 2.\]
This yields a contradiction, since there is at most one vertex with adjusted weight greater than one by Condition~\ref{it:largeweight}. Thus $w'(v_0)=w'(v_{k+1})=1$. Hence, Lemma~\ref{lem:0-chain_rigid} implies that the subgraph of $\Gamma$ given by $v_0,\dots, v_{k+1}$ is rigid. Thus if $u$ takes the form $u=\pm(e_0+ \dots + e_k)$, then it pairs nontrivially with both $v_{0}$ and $v_{k+1}$, which is impossible. Thus we conclude $$w(u)\geq k+2\geq 4\qquad \text{ and } \qquad w'(u)\geq 2,$$ allowing us to obtain \ref{it:u_large_norm} of the lemma.

If $u$ has degree two, then as in the previous cases it is adjacent to a chain of twos and this chain is distinct from $v_1, \dots, v_k$, since $u$ is not adjacent to $v_1$ or $v_k$. If $u$ has degree $d(u)<2$, then, we have
$$ 3 \geq w'(u) \geq w(u) -1 \geq k+1 \geq 3.$$
In particular, we have $w'(u)=3$. Since $u$ cannot be adjacent to both $v_0$ and $v_{k+1}$, we may assume without loss of generality that $u$ is not adjacent to $v_0$. Since $w(v_0)-2=w'(v_0)=1$, we see that $v_0$ is adjacent to a vertex $v_{-1}\neq u$ which has adjusted weight $w'(v_{-1})\leq 1$ by Condition~\ref{it:largeweight}.  Furthermore, since $w'(u)=3$, Condition~\ref{it:weight3condition} implies that $\Gamma$ cannot contain two adjacent vertices of adjusted weight one. Thus $w'(v_{-1})=0$ and we can apply Lemma~\ref{lem:degreetwochains} to obtain a second chain of twos in $\Gamma$ and hence \ref{it:min_candidate} of the lemma in this case.

\begin{case}
Neither $v_1$ nor $v_k$ is a leaf and $u$ is adjacent to $v_1$ or $v_k$. 
\end{case}
Without loss of generality, we may assume that $u$ is adjacent to $v_1$. Hence, we have $w'(u)\geq 1$ and $ w(u)\geq 3$. Since $v_{k}$ is not a leaf, there is a vertex $v_{k+1}$, which is adjacent to $v_{k}$ and satisfies $w'(v_{k+1})\geq 1$ and $w(v_{k+1})\geq 3$. If we had both $w'(u)=1$ and $w'(v_{k+1})=1$, then Lemma~\ref{lem:0-chain_rigid} would imply that the subgraph given by $u, v_1, \dots, v_{k+1}$ is rigid, however the hypothesis on $u$ shows that this subgraph is not rigid. Thus we have that $w'(u)>1$ or $w'(v_{k+1})>1$. In either case, Conditions~\ref{it:largeweight} and \ref{it:long_chain} imply that $k\geq 3$.

Now, in order to obtain the correct pairings with $v_1,\dots, v_k$, the vertex $u$ must take the form
\[
u= a(e_0 + \dots + e_k)+e_0 + u'
\]
where $a$ is a nonzero integer and $u'\cdot e_0 = \dots = u'\cdot e_k=0$. Since $k\geq 3$, it is evident that \ref{it:supp_bound} of the lemma holds. 

Computing the norm of $u$ shows $w(u)\geq k$ with equality if and only if $u=-(e_1 + \dots + e_k)$. That is, we have equality only if $a=-1$ and $u'=0$. However, we will see that $u$ cannot take the form $u=-(e_1 + \dots + e_k)$. In order to obtain the correct pairings with $v_1,\dots, v_k$, the vertex $v_{k+1}$ must take the form
\[
v_{k+1}= b(e_0 + \dots + e_k)-e_k + v'
\]
where $b$ is an integer and $v'\cdot e_0 = \dots = v'\cdot e_k=0$. If $u=-(e_1 + \dots + e_k)$, then we have that $u\cdot v_{k+1}=-kb+1=0$, which is absurd. Thus we conclude that $w(u)\geq k+1\geq 4$. This implies that $w'(u)\geq 2$, giving \ref{it:u_large_norm} of the lemma. Furthermore, Condition~\ref{it:largeweight} implies that $w'(v_{k+1})=1$, $w(v_{k+1})=3$, and $d(v_{k+1})=2$.

Finally, we wish to show that there is an additional chain of twos in $\Gamma$. If $w'(u)=3$, then Condition~\ref{it:weight3condition} implies that $\Gamma$ does not contain adjacent vertices with adjusted weight one. Thus $v_{k+1}$ must be adjacent to vertices of adjusted weight zero on both sides, giving a second chain of twos distinct from $v_1,\dots, v_{k}$. Thus it remains to assume that $w'(u)=2$. This implies that $w(u)=4$, $d(u)=2$, and $k=3$. Now, since both $u$ and $v_{k+1}$ have degree two, they must be adjacent to two other vertices. If $\Gamma$ contains no further chains of weight two, then Condition~\ref{it:3adjacent} implies that the connected component of $\Gamma$ containing $u$ must take the form

$$\begin{tikzpicture}[xscale=1.0,yscale=1,baseline={(0,0)}]
    \node at (-1.0, .4) {$2$};
    \node at (0.0, .4) {$4$};
    \node at (1.0, .4) {$2$};
    \node at (2.0, .4) {$2$};
    \node at (3.0, .4) {$2$};
        \node at (4.0, .4) {$3$};
                \node at (5.0, .4) {$2$};
        \node (A1_9) at (-1, 0) {$\bullet$};
    \node (A1_0) at (0, 0) {$\bullet$};
    \node (A1_1) at (1, 0) {$\bullet$};
    \node (A1_2) at (2, 0) {$\bullet$};
    \node (A1_3) at (3, 0) {$\bullet$};
        \node (A1_4) at (4, 0) {$\bullet$};
                \node (A1_5) at (5, 0) {$\bullet$};
        \path (A1_9) edge [-] node [auto] {$\scriptstyle{}$} (A1_0);
    \path (A1_0) edge [-] node [auto] {$\scriptstyle{}$} (A1_1);
    \path (A1_1) edge [-] node [auto] {$\scriptstyle{}$} (A1_2);
        \path (A1_2) edge [-] node [auto] {$\scriptstyle{}$} (A1_3);
                \path (A1_3) edge [-] node [auto] {$\scriptstyle{}$} (A1_4);
                                \path (A1_4) edge [-] node [auto] {$\scriptstyle{}$} (A1_5);
  \end{tikzpicture}.$$
One can verify by computer calculation that this particular graph is rigid (see Appendix~\ref{sec:computer} and Lemma~\ref{lem:computer2} for further details). In particular, it implies that  $|\supp(u) \cap \supp(v_1, \dots, v_k)|=1$ contradicting the assumption on $u$. This establishes \ref{it:min_candidate} in the final case and concludes the proof.\end{proof}

Finally we are ready to show that the Working Conditions imply rigidity.
\begin{thm}\label{maintechnical:thm}
If $\Gamma$ is a disjoint union of linear plumbing graphs satisfying the Working Conditions, then $\Gamma$ is rigid.
\end{thm}

\begin{proof}
We prove this by induction on the number of vertices in $\Gamma$. For the base case we assume that $\Gamma$ contains at most two vertices. It follows from Condition~\ref{it:largeweight} that $\Gamma$ must be one of the following five graphs in this case:
\[\begin{tikzpicture}[xscale=1.0,yscale=1,baseline={(0,0)}]
    \node at (1, .4) {$2$};
    \node (A_1) at (1, 0) {$\bullet$};
  \end{tikzpicture},\qquad
  \begin{tikzpicture}[xscale=1.0,yscale=1,baseline={(0,0)}]
    \node at (1, .4) {$3$};
    \node (A_1) at (1, 0) {$\bullet$};
  \end{tikzpicture},\qquad
  \begin{tikzpicture}[xscale=1.0,yscale=1,baseline={(0,0)}]
    \node at (1, .4) {$2$};
    \node at (2, .4) {$2$};
    \node (A_1) at (1, 0) {$\bullet$};
    \node (A_2) at (2, 0) {$\bullet$};
        \path (A_1) edge [-] node [auto] {$\scriptstyle{}$} (A_2);
  \end{tikzpicture},\qquad
  \begin{tikzpicture}[xscale=1.0,yscale=1,baseline={(0,0)}]
    \node at (1, .4) {$2$};
    \node at (2, .4) {$3$};
    \node (A_1) at (1, 0) {$\bullet$};
    \node (A_2) at (2, 0) {$\bullet$};
        \path (A_1) edge [-] node [auto] {$\scriptstyle{}$} (A_2);
  \end{tikzpicture},\qquad \text{ or } \qquad
   \begin{tikzpicture}[xscale=1.0,yscale=1,baseline={(0,0)}]
    \node at (1, .4) {$2$};
    \node at (2, .4) {$4$};
    \node (A_1) at (1, 0) {$\bullet$};
    \node (A_2) at (2, 0) {$\bullet$};
        \path (A_1) edge [-] node [auto] {$\scriptstyle{}$} (A_2);
  \end{tikzpicture}.\]
  Each of these is easily checked to be rigid, giving the base case. 
  
  Now suppose that $\Gamma$ contains at least three vertices. We will see that such a $\Gamma$ must contain a chain of twos. If $\Gamma$ has at least three vertices in one connected component, then it contains a vertex of degree two guaranteeing a chain of twos by Lemma~\ref{lem:degreetwochains}. On the other hand if $\Gamma$ contains only connected components with at most two vertices, then Condition~\ref{it:largeweight} implies that only one of these components contains a vertex with $w'(v)>1$. Thus all but one of these components will take the following form
  $$\begin{tikzpicture}[xscale=1.0,yscale=1,baseline={(0,0)}]
    \node at (1, .4) {$2$};
    \node at (2, .4) {$2$};
    \node (A_1) at (1, 0) {$\bullet$};
    \node (A_2) at (2, 0) {$\bullet$};
        \path (A_1) edge [-] node [auto] {$\scriptstyle{}$} (A_2);
  \end{tikzpicture}.$$
Thus we can assume that $\Gamma$ contains a chain of twos.

We claim further that $\Gamma$ contains a chain of twos $v_1, \dots, v_k$ such that for every vertex $v\not\in\{v_1,\dots, v_k\}$ we have
$$ |\supp(v) \cap \supp(v_1, \dots, v_k)|\leq 1.$$
Let $u_1, \dots, u_{\ell}$ be a chain of twos. Suppose that there is a vertex $u$ not contained in the chain of twos $u_1, \dots, u_{\ell}$ with 
$$ |\supp(u) \cap \supp(u_1, \dots, u_\ell)|\geq 2.$$
 By Lemma~\ref{lem:min_example_properties}, we see that $u$ is the unique vertex with $w'(u)>1$ and that $ |\supp(u) \cap \supp(u_1, \dots, u_\ell)|\geq 3$. Moreover, $\Gamma$ must contain some other chain of twos $v_1, \dots, v_{k}$. Now, suppose there is a vertex $v$ not contained in the chain of twos $v_1, \dots, v_{k}$ such that \[|\supp(v) \cap \supp(v_1, \dots, v_{k})|\geq 2.\]
Again by Lemma~\ref{lem:min_example_properties}, we have that $v$ is the unique vertex with $w'(v)>1$ and that $$ |\supp(v) \cap \supp(v_1, \dots, v_{k})|\geq 3.$$ In particular, we conclude that $u=v$. However, as $\supp(v_1, \dots, v_{k})$ and $\supp(u_1, \dots, u_\ell)$ are disjoint, this implies that
\[
w(v)\geq \supp(v)\geq 6,
\]
contradicting Condition~\ref{it:largeweight}.  Hence, we conclude that 
$$|\supp(v) \cap \supp(v_1, \dots, v_{k})|\leq 1$$ for all vertices $v$ not contained in the chain of twos $v_1, \dots, v_k$.

Now, take such chain of twos $v_1, \dots, v_k$ and assume that $v_i$ is embedded as $v_i=e_i-e_{i-1}$ for each $i=1, \dots, k$. We break down the proof into three cases.
\setcounter{case}{0}
\begin{case}
Both $v_1$ and $v_k$ are leaves.
\end{case}
In this case, $v_1, \dots, v_k$ forms a connected component of $\Gamma$ and for every other vertex we have $\supp(v) \cap \supp(v_1, \dots, v_k)=\varnothing$. Let $\Gamma'$ be the plumbing graph obtained by deleting this component. Note that $\Gamma'$ also satisfies the Working Conditions. The embedding of $\Gamma$ restricts to give an embedding of $\Gamma'$. By the inductive hypothesis $\Gamma'$ is rigid and hence the induced embedding is standard. It follows that the original embedding of $\Gamma$ was standard. By Remark~\ref{rem:standard_embeddings}, this implies that $\Gamma$ is rigid.
\begin{case}

Exactly one of $v_1$ and $v_k$ is a leaf.
\end{case}
In this case, we can assume that there is a vertex $v_{k+1}$, which is adjacent to $v_k$. This vertex must take the form $v_{k+1}=-e_k+v'$ where $v'\cdot e_0 =\dots = v'\cdot e_k=0$. For all other vertices we must have $\supp(v) \cap \supp(v_1, \dots, v_k)=\varnothing$. Now, let $\Gamma'$ be the plumbing graph obtained by deleting $v_1, \dots, v_k$ and decreasing the weight of $v_{k+1}$ by one. Note that $\Gamma'$ also satisfies the Working Conditions. The plumbing graph $\Gamma'$ inherits an embedding in $\Z^n$ obtained by deleting $v_1,\dots, v_k$ and sending $v_{k+1}$ into $\Z^n$ as $v'$. By the inductive hypothesis $\Gamma'$ is rigid and hence the induced embedding is standard. It follows that the original embedding of $\Gamma$ was standard. By Remark~\ref{rem:standard_embeddings}, this implies that $\Gamma$ is rigid.
\begin{case}

Neither $v_1$ nor $v_k$ is a leaf.
\end{case}
In this case, there are vertices $v_0$ and $v_{k+1}$ adjacent to $v_1$ and $v_k$ respectively. These must embed as $v_{0}=e_0+v'$ and $v_{k+1}=-e_k+v''$, where $v'$ and $v''$ satisfy $v'\cdot e_0 =\dots = v'\cdot e_k=0$ and $v''\cdot e_0 =\dots = v''\cdot e_k=0$. Take $\Gamma'$ to be the plumbing graph  obtained by deleting $v_1, \dots, v_k$ and decreasing the weights on $v_0$ and $v_{k+1}$ by one. Note that $\Gamma'$ also satisfies the Working Conditions. This inherits an embedding in $\Z^n$ obtained by deleting $v_1,\dots, v_k$ and embedding $v_0$ and $v_{k+1}$ as $v'$ and $v''$. By the inductive hypothesis $\Gamma'$ is rigid and hence the induced embedding is standard. It follows that the original embedding of $\Gamma$ was standard. By Remark~\ref{rem:standard_embeddings}, this implies that $\Gamma$ is rigid. This completes the induction step.
\end{proof}
This allows us to conclude with the proof of Theorem~\ref{thm:technical}, which forms the final ingredient in the proof of Theorem~\ref{thm:main}.
\begin{proof}[Proof of Theorem~\ref{thm:technical}]
By Lemma~\ref{lem:forbidden_configurations}, the manifold $-\natural X(p_i, p_i-q_i)$ is plumbed according to a dual plumbing graph $P^*$ satisfying the Working Conditions. By Theorem~\ref{maintechnical:thm}, this implies the plumbing graph $P^*$ is rigid and hence the intersection form  of $-\natural X(p_i, p_i-q_i)$  is also rigid.
\end{proof}

\section{Lens spaces not bounding small definite manifolds}\label{sec:no_small}
In this section, we use ideas from the previous sections to prove Theorem~\ref{thm:no_small_fillings} which asserts the existence of lens spaces that do not admit a ``small'' negative-definite filling with either orientation. As in previous sections we use $\Z^N$ to denote the standard positive-definite diagonal lattice $(\Z^N, \langle 1 \rangle^N)$.

Throughout this section we will use $L_{m,n}$ to denote the lens space $L(p,q)$, where $p/q$ is given by the continued fraction expansion
\[
\frac{p}{q}=\left[9^{[n]},(3,2,2,2,2)^{[m]}\right]^-
\]
and $m$ and $n$ are positive integers.
The examples in Theorem~\ref{thm:no_small_fillings} will be obtained by choosing suitable examples among the $L_{m,n}$.
By \eqref{eq:Riemanschneider_long}, we see that $p/(p-q)$ satisfies the continued fraction
\begin{equation}\label{eq:dual_frac}
\frac{p}{p-q}=\left[2,(2,2,2,2,2,2,3)^{[n]},7^{[m-1]},6\right]^-.
\end{equation}

First we obtain a lower bound on any smooth negative-definite filling of $-L_{m,n}$.
\begin{prop}\label{prop:first_bound}
If $X$ is a smooth negative-definite filling of $-L_{m,n}$, then $b_2(X)\geq m-n+1$.
\end{prop} 
\begin{proof}
Let $X$ be a smooth negative-definite filling of $-L_{m,n}$. We construct a smooth closed positive-definite 4-dimensional manifold $-X\cup_{L(p,q)} -X(p,q)$. By Donaldson's diagonalization theorem this has diagonalizable intersection form and so we obtain a map of lattices
\[
-Q_{X(p,q)}\hookrightarrow \Z^{N},
\]

where $N= b_2(X(p,q)) + b_2(X)$. Here, $-X(p,q)$ is the 4-manifold obtained by plumbing according to the  plumbing graph $\Gamma$:
$$\begin{tikzpicture}[xscale=1.0,yscale=1,baseline={(0,0)}]
    \node at (1, .4) {$a_1$};
    \node at (2, .4) {$a_2$};
        \node at (3, .4) {$a_3$};
    \node at (5, .4) {$a_{\ell}$};
    \node (A1_1) at (1, 0) {$\bullet$};
    \node (A1_2) at (2, 0) {$\bullet$};
    \node (A1_3) at (3, 0) {$\bullet$};
    \node (A1_4) at (4, 0) {$\cdots$};
    \node (A1_5) at (5, 0) {$\bullet$};
    \path (A1_2) edge [-] node [auto] {$\scriptstyle{}$} (A1_3);
    \path (A1_3) edge [-] node [auto] {$\scriptstyle{}$} (A1_4);
        \path (A1_4) edge [-] node [auto] {$\scriptstyle{}$} (A1_5);
    \path (A1_1) edge [-] node [auto] {$\scriptstyle{}$} (A1_2);
  \end{tikzpicture},$$
  where the coefficients are given by the tuple
  \[
 (a_1,\dots, a_\ell)= \left(9^{[n]},(3,2,2,2,2)^{[m]}\right).
  \]
Thus the map of lattices $-Q_{X(p,q)}\hookrightarrow \Z^{N}$ induces an embedding of $\Gamma$ into $\Z^N$. Let $\Gamma'$ be the subgraph of $\Gamma$ corresponding to the tuple $(3,2,2,2,2)^{[m]}$. One can check that $\Gamma'$ satisfies the Working Conditions from Section~\ref{Section:preliminaries}. Therefore $\Gamma'$ is rigid by Theorem~\ref{maintechnical:thm}. By the bound in Remark~\ref{rem:standard_embeddings}\ref{it:standard_bound} this implies that $N\geq 6m+1$ (the graph $\Gamma'$ has $5m-1$ edges and the sum of its weights is $11m$). Since $b_2(X(p,q))=n+5m$, this gives 
\[b_2(X)=N-b_2(X(p,q))\geq m-n+1,\] which is the desired bound on $b_2(X)$.
\end{proof}

Next we obtain a lower bound on any smooth negative-definite filling of $L_{m,n}$. The argument is similar although a little more delicate.
\begin{prop}\label{prop:second_bound}
If $X$ is a smooth negative-definite filling of $L_{m,n}$, then $b_2(X)\geq n-1$.
\end{prop}
\begin{proof}
Let $X$ be a smooth negative-definite filling of $L_{m,n}$. By the fact that $-L_{m,n}\cong L(p,p-q)$ and \eqref{eq:dual_frac}, we have that $-X(p,p-q)$ is the 4-manifold obtained by plumbing according to the plumbing graph~$\Gamma$:
$$\begin{tikzpicture}[xscale=1.0,yscale=1,baseline={(0,0)}]
    \node at (1, .4) {$a_1$};
    \node at (2, .4) {$a_2$};
        \node at (3, .4) {$a_3$};
    \node at (5, .4) {$a_{\ell}$};
    \node (A1_1) at (1, 0) {$\bullet$};
    \node (A1_2) at (2, 0) {$\bullet$};
    \node (A1_3) at (3, 0) {$\bullet$};
    \node (A1_4) at (4, 0) {$\cdots$};
    \node (A1_5) at (5, 0) {$\bullet$};
    \path (A1_2) edge [-] node [auto] {$\scriptstyle{}$} (A1_3);
    \path (A1_3) edge [-] node [auto] {$\scriptstyle{}$} (A1_4);
        \path (A1_4) edge [-] node [auto] {$\scriptstyle{}$} (A1_5);
    \path (A1_1) edge [-] node [auto] {$\scriptstyle{}$} (A1_2);
  \end{tikzpicture},$$
  where the coefficients are given by the tuple
\[
(a_1,\dots, a_\ell)=\left(2,(2,2,2,2,2,2,3)^{[n]},7^{[m-1]},6\right).
\]
In particular, we see that $b_2(X(p,p-q))=m+7n+1$. We construct a smooth closed positive-definite 4-dimensional manifold $-X \cup_{L(p,p-q)} -X(p,p-q)$. By Donaldson's diagonalization theorem this has diagonalizable intersection form and so we obtain an embedding of $\Gamma$ into $\Z^N$ where $N= b_2(X(p,p-q)) + b_2(X)$.
We now separate the vertices of $\Gamma$ into two groups. We will use $v_1,\dots, v_{7n}$ to denote the first $7n$ vertices, that is the vertices corresponding to the weights:
\[
\left(2,(2,2,2,2,2,2,3)^{[n-1]},2,2,2,2,2,2\right)
\]
and $u_0, \dots, u_{m}$ to denote the remaining $m+1$ vertices, that is the vertices corresponding to the weights:
\[
\left(3,7^{[m-1]},6\right).
\]
The subgraph induced by the vertices $v_1,\dots, v_{7n}$ satisfies the Working Conditions from Section~\ref{Section:preliminaries} and is hence rigid by Theorem~\ref{maintechnical:thm}. Thus up to automorphisms of $\Z^N$ the vertices $v_1,\dots, v_{7n}$ admit unique embedding and this is a standard embedding. Moreover, using the formula derived in Remark~\ref{rem:standard_embeddings}\ref{it:standard_bound} this unique embedding uses $8n$ distinct orthogonal unit vectors. Thus we can assume that the first $7n$ vertices are embedded in $v_1, \dots, v_{7n}\subseteq \langle e_1, \dots, e_{8n}\rangle$. 
\begin{claim}
Any nonzero vector 
\[v\in \langle v_1, \dots, v_{7n}\rangle^\bot \subseteq \langle e_1, \dots, e_{8n}\rangle\]
has length $\norm{v}\geq 8$.
\end{claim}
\begin{proof}[Proof of Claim]
We will show that 
\[\langle v_1, \dots, v_{7n}\rangle^\bot \subseteq \langle e_1, \dots, e_{8n}\rangle\]
is isomorphic to the lattice $Q_P$ associated to the plumbing graph $P$ given by
$$\begin{tikzpicture}[xscale=1.0,yscale=1,baseline={(0,0)}]
    \node at (1, .4) {$9$};
        \node at (3, .4) {$9$};
    \node at (4, .4) {$8$};
    \node (A1_1) at (1, 0) {$\bullet$};
    \node (A1_2) at (2, 0) {$\cdots$};
    \node (A1_3) at (3, 0) {$\bullet$};
    \node (A1_4) at (4, 0) {$\bullet$};
    \path (A1_2) edge [-] node [auto] {$\scriptstyle{}$} (A1_3);
    \path (A1_3) edge [-] node [auto] {$\scriptstyle{}$} (A1_4);
    \path (A1_1) edge [-] node [auto] {$\scriptstyle{}$} (A1_2);
  \end{tikzpicture}.$$
Then Lemma~\ref{lem:def_calc} implies that every nonzero element $x$ of $Q_P$ satisfies $Q_P(x,x)\geq 8$. Thus this isomorphism is sufficient to prove the bound. 
Let $r/s\in \Q$ be defined by the continued fraction
\[
r/s=\left[2,(2,2,2,2,2,2,3)^{[n-1]},2,2,2,2,2,2\right]^-.
\]
Since the plumbing graph for $-X(r,s)$ satisfies the Working Conditions, Theorem~\ref{maintechnical:thm} shows that the intersection form of $-X(r,s)$ is rigid. On the other hand, $r/(r-s)$ has continued fraction
\[
r/(r-s)=\left[[9]^{n-1},8\right]^-,
\]
and so the intersection form of $-X(r,r-s)$ is isomorphic to $Q_P$.

Now consider the smooth closed positive-definite 4-manifold
\[
Z=-X(r,s) \cup -X(r,r-s).
\]
By Donaldson's diagonalization theorem and the fact that $b_2(Z)=b_2(X(r,s))+b_2(X(r,r-s))=8n$, we have $Q_Z\cong \Z^{8n}$.
By Lemma~\ref{lem:orthogonal_complement}, the intersection form of $-X(r,r-s)$ is isomorphic to the orthogonal complement of $Q_{-X(r,s)}$ in $Q_Z\cong \Z^{8n}$. However, the rigidity of $Q_{X(r,s)}$ shows that up to automorphism the vectors $v_1,\dots, v_{7n}$ form a basis for the image of $Q_{-X(r,s)}$ in $\Z^{8n}$. Putting together all these facts we obtain
\[
Q_P\cong Q_{-X(r,r-s)}\cong \langle v_1, \dots, v_{7n}\rangle^\bot \subseteq \Z^{8n},
\]
as desired.
\end{proof}
 Since the vertices $u_1, \dots, u_{m}$ satisfy $\norm{u_i}\leq 7$ and are orthogonal to $v_1, \dots, v_{7n}$, the above claim shows that they all must have trivial pairing with the each of $e_1,\dots, e_{8n}$. This implies that the embeddings of $u_1, \dots, u_m$ must use at least $m$ additional unit vectors, and we have that $N\geq 8n+m$. Thus,
\[
b_2(X)=N-b_2(X(p,p-q))\geq 8n+m -(m+7n+1)=n-1,
\]
which is the required bound.
\end{proof}

With these bounds in hand we exhibit the examples for Theorem~\ref{thm:no_small_fillings}.
\begin{proof}[Proof of Theorem~\ref{thm:no_small_fillings}]
We take $L_k$ to be the lens space $L_{m,n}$ with $m=2k$ and $n=k+1$. By Proposition~\ref{prop:first_bound}, we see that a negative-definite filling $X$ of $-L_k$ satisfies $b_2(X)\geq m-n+1=k$ and, by Proposition~\ref{prop:second_bound}, a negative-definite filling $X$ of $L_k$ satisfies $b_2(X)\geq n-1=k$.
\end{proof}

\section{Embedding lens spaces in 4-manifolds}\label{sec:embeddings}
We now turn our attention to embedding lens spaces in 4-manifolds.
\begin{prop}\label{prop:indef_embedding}
Let $L_k$ be a lens space as given in Theorem~\ref{thm:no_small_fillings}. If $L_k$ embeds as a separating submanifold of a smooth closed 4-manifold $X$ with $|\sigma(X)|\geq b_2(X)-2$, then $k\leq b_2(X)$.
\end{prop}
\begin{proof}
By hypothesis and Theorem~\ref{thm:no_small_fillings}, any smooth negative-definite filling of $L_k$ or $-L_k$ has $b_2\geq k$.
Suppose that $L_k$ smoothly embeds in a closed $4$-manifold $X$ as a separating submanifold and that $|\sigma(X)|\geq b_2(X)-2$. By changing the orientation of $X$ if necessary we may assume that $\sigma(X)\leq 2-b_2(X)$ and hence that $b_2^+(X)\leq 1$. The embedding of $L_k$ in $X$ gives a decomposition of $X$ into two pieces $X_1$ and $X_2$ with $\partial X_1 \cong -\partial X_2 \cong L_k$. Moreover, by Novikov additivity and the fact that a lens space is a rational homology sphere, we have that $b_2^+(X)=b_2^+(X_1)+b_2^+(X_2)$ and $b_2^-(X)=b_2^-(X_1)+b_2^-(X_2)$. Thus, we have that at least one of $b_2^+(X_1)$ or $b_2^+(X_2)$ is zero. That is, at least one of $X_1$ or $X_2$ is negative-definite. This implies that $b_2(X_1)\geq k$ or $b_2(X_2)\geq k$ and hence that $b_2(X)\geq k$.
\end{proof}

Mimicking arguments that can be found in \cite{Aceto-Golla-Larson:2017-1}, we note that there is no spin 4-manifold that contains every lens space as a separating submanifold.
\begin{prop}\label{prop:spin}
If the lens space $L(n,n-1)$ with $n\geq 1$ odd smoothly embeds in a closed spin 4-manifold $X$ as a separating submanifold, then
\[
n\leq 9b_2(X)+1.
\]
\end{prop}
\begin{proof}
Note that since $n$ is odd, the lens space $L(n,n-1)$ admits a unique spin structure. Cutting $X$ along the embedded $L(n,n-1)$ yields a spin 4-manifold $Z$ filling $-L(n,n-1)$. On the other hand, the manifold $X(n,n-1)$ is negative-definite and spin with $b_2(X(n,n-1))=n-1$. Since the spin structure on $L(n,n-1)$ is unique, the closed 4-manifold $W=X(n,n-1) \cup_{L(n,n-1)} Z$ is spin. We have
\[b_2(W)=b_2(Z)+n-1\leq b_2(X)+n-1\]
and
\[|\sigma(W)|=|\sigma(Z)-n+1|\geq n-1-b_2(X).\]
Thus by Furuta's theorem \cite{Furuta:2001-1}, we have
\[
b_2(X)+n-1\geq b_2(W) \geq \frac54 |\sigma(W)|\geq \frac54(n-1 -b_2(X)).
\] 
Rearranging implies that $n\leq 9b_2(X)+1$.
\end{proof}

\begin{proof}[Proof of Theorem~\ref{thm:embedding}]
Let $X$ be a smooth closed 4-manifold that contains every lens space as a smoothly embedded separating submanifold. Proposition~\ref{prop:spin} implies that $X$ cannot be spin. Proposition~\ref{prop:indef_embedding} implies that $|\sigma(X)|< b_2(X)-2$. However $b_2(X)\equiv \sigma(X) \bmod 2$, so this implies that $|\sigma(X)|\leq b_2(X)-4$.\end{proof}

\bibliographystyle{alpha}
\bibliography{bib}
\appendix

\section{Computer calculations}\label{sec:computer}
We compile the results of the calculations necessary to complete the proofs of Lemma~\ref{lem:0-chain_rigid} and Lemma~\ref{lem:min_example_properties}. Although these calculations could be carried out by hand, the details would be so tedious that computer calculation is certainly preferable. In both cases, we are required to classify all embeddings of certain linear plumbing graphs into $\Z^n$ for all $n$.  Let $\Gamma$ be a linear plumbing graph of the form
$$\begin{tikzpicture}[xscale=1.0,yscale=1,baseline={(0,0)}]
    \node at (1, .4) {$a_1$};
    \node at (2, .4) {$a_2$};
        \node at (3, .4) {$a_3$};
    \node at (5, .4) {$a_{m}$};
    \node (A1_1) at (1, 0) {$\bullet$};
    \node (A1_2) at (2, 0) {$\bullet$};
    \node (A1_3) at (3, 0) {$\bullet$};
    \node (A1_4) at (4, 0) {$\cdots$};
    \node (A1_5) at (5, 0) {$\bullet$};
    \path (A1_2) edge [-] node [auto] {$\scriptstyle{}$} (A1_3);
    \path (A1_3) edge [-] node [auto] {$\scriptstyle{}$} (A1_4);
        \path (A1_4) edge [-] node [auto] {$\scriptstyle{}$} (A1_5);
    \path (A1_1) edge [-] node [auto] {$\scriptstyle{}$} (A1_2);
  \end{tikzpicture}.$$
If we take $M$ to be the matrix

$$M=\begin{pmatrix}
a_1 & -1&  0     &0  \\
-1  &a_2& -1     &0  \\
0   &-1          &\ddots &  -1 \\
0    &  0        &  -1 & a_m \\
      \end{pmatrix},
      $$
then embeddings of $\Gamma$ in $\Z^n$ correspond to matrix factorizations of the form $A(A^T)=M$, where $A$ is an $m \times n$ integer matrix. Explicitly, the matrix $A$ corresponds to the embedding where the $i$th vertex $v_i$ is embedded as 
\[
v_i = a_{i1} e_1 + \dots + a_{in}e_n,
\]
and $a_{ij}\in \Z$ represents the $ij$ entry of $A$. Under this correspondence, automorphisms of $\Z^n$ correspond to permuting the columns of $A$ and multiplying columns by $-1$.

The \texttt{OrthogonalEmbeddings} command in GAP takes a symmetric integer matrix $M$ and produces, up to permutation of columns and multiplication of columns by $-1$, all integer matrices $A$ such that $A (A^T)=M$ and all columns of $A$ are nonzero \cite{GAP}. Thus given a plumbing graph $\Gamma$, we can use the \texttt{OrthogonalEmbeddings} to study all possible embeddings of $\Gamma$ into diagonal lattices.

The following lemma is necessary to justify Lemma~\ref{lem:0-chain_rigid}. 
\begin{lem}\label{lem:computer1}
For each of the following linear plumbing graphs, the subgraph comprising the vertices with red/circled weights is rigid:
\begin{enumerate}[label=(\roman*), font=\upshape]
\item $\begin{tikzpicture}[xscale=1.0,yscale=1,baseline={(0,0)}]
    \node at (1, .4) {$2$};
    \node at (2, .4) {$\red{\circled{\emph{3}}}$};
    \node at (3, .4) {$\red{\circled{\emph{2}}}$};
    \node at (4, .4) {$\red{\circled{\emph{2}}}$};
        \node at (5, .4) {$\red{\circled{\emph{3}}}$};
                \node at (6, .4) {$2$};
    \node (A1_1) at (1, 0) {$\bullet$};
    \node (A1_2) at (2, 0) {$\bullet$};
    \node (A1_3) at (3, 0) {$\bullet$};
    \node (A1_4) at (4, 0) {$\bullet$};
        \node (A1_5) at (5, 0) {$\bullet$};
            \node (A1_6) at (6, 0) {$\bullet$};
    \path (A1_2) edge [-] node [auto] {$\scriptstyle{}$} (A1_3);
    \path (A1_3) edge [-] node [auto] {$\scriptstyle{}$} (A1_4);
    \path (A1_1) edge [-] node [auto] {$\scriptstyle{}$} (A1_2);
        \path (A1_4) edge [-] node [auto] {$\scriptstyle{}$} (A1_5);
            \path (A1_5) edge [-] node [auto] {$\scriptstyle{}$} (A1_6);
  \end{tikzpicture}$

\item $\begin{tikzpicture}[xscale=1.0,yscale=1,baseline={(0,0)}]
    \node at (1, .4) {$2$};
    \node at (2, .4) {$\red{\circled{\emph{3}}}$};
    \node at (3, .4) {$\red{\circled{\emph{2}}}$};
    \node at (4, .4) {$\red{\circled{\emph{2}}}$};
        \node at (5, .4) {$\red{\circled{\emph{3}}}$};
\node at (6, .4) {$3$};
\node at (7, .4) {$2$};
                                
    \node (A1_1) at (1, 0) {$\bullet$};
    \node (A1_2) at (2, 0) {$\bullet$};
    \node (A1_3) at (3, 0) {$\bullet$};
    \node (A1_4) at (4, 0) {$\bullet$};
        \node (A1_5) at (5, 0) {$\bullet$};
            \node (A1_6) at (6, 0) {$\bullet$};
\node (A1_7) at (7, 0) {$\bullet$};

    \path (A1_2) edge [-] node [auto] {$\scriptstyle{}$} (A1_3);
    \path (A1_3) edge [-] node [auto] {$\scriptstyle{}$} (A1_4);
    \path (A1_1) edge [-] node [auto] {$\scriptstyle{}$} (A1_2);
        \path (A1_4) edge [-] node [auto] {$\scriptstyle{}$} (A1_5);
\path (A1_5) edge [-] node [auto] {$\scriptstyle{}$} (A1_6);
\path (A1_6) edge [-] node [auto] {$\scriptstyle{}$} (A1_7);
  \end{tikzpicture}$
  
\item $\begin{tikzpicture}[xscale=1.0,yscale=1,baseline={(0,0)}]
    \node at (1, .4) {$2$};
    \node at (2, .4) {$2$};
    \node at (3, .4) {$\red{\circled{\emph{3}}}$};
    \node at (4, .4) {$\red{\circled{\emph{2}}}$};
        \node at (5, .4) {$\red{\circled{\emph{2}}}$};
\node at (6, .4) {$\red{\circled{\emph{3}}}$};
\node at (7, .4) {$3$};
                                
    \node (A1_1) at (1, 0) {$\bullet$};
    \node (A1_2) at (2, 0) {$\bullet$};
    \node (A1_3) at (3, 0) {$\bullet$};
    \node (A1_4) at (4, 0) {$\bullet$};
        \node (A1_5) at (5, 0) {$\bullet$};
            \node (A1_6) at (6, 0) {$\bullet$};
\node (A1_7) at (7, 0) {$\bullet$};

    \path (A1_2) edge [-] node [auto] {$\scriptstyle{}$} (A1_3);
    \path (A1_3) edge [-] node [auto] {$\scriptstyle{}$} (A1_4);
    \path (A1_1) edge [-] node [auto] {$\scriptstyle{}$} (A1_2);
        \path (A1_4) edge [-] node [auto] {$\scriptstyle{}$} (A1_5);
\path (A1_5) edge [-] node [auto] {$\scriptstyle{}$} (A1_6);
\path (A1_6) edge [-] node [auto] {$\scriptstyle{}$} (A1_7);
  \end{tikzpicture}$
\item $\begin{tikzpicture}[xscale=1.0,yscale=1,baseline={(0,0)}]
    \node at (1, .4) {$2$};
    \node at (2, .4) {$2$};
    \node at (3, .4) {$\red{\circled{\emph{3}}}$};
    \node at (4, .4) {$\red{\circled{\emph{2}}}$};
        \node at (5, .4) {$\red{\circled{\emph{2}}}$};
\node at (6, .4) {$\red{\circled{\emph{3}}}$};
\node at (7, .4) {$4$};
\node at (8, .4) {$2$};
                                
    \node (A1_1) at (1, 0) {$\bullet$};
    \node (A1_2) at (2, 0) {$\bullet$};
    \node (A1_3) at (3, 0) {$\bullet$};
    \node (A1_4) at (4, 0) {$\bullet$};
        \node (A1_5) at (5, 0) {$\bullet$};
            \node (A1_6) at (6, 0) {$\bullet$};
\node (A1_7) at (7, 0) {$\bullet$};
\node (A1_8) at (8, 0) {$\bullet$};

    \path (A1_2) edge [-] node [auto] {$\scriptstyle{}$} (A1_3);
    \path (A1_3) edge [-] node [auto] {$\scriptstyle{}$} (A1_4);
    \path (A1_1) edge [-] node [auto] {$\scriptstyle{}$} (A1_2);
        \path (A1_4) edge [-] node [auto] {$\scriptstyle{}$} (A1_5);
\path (A1_5) edge [-] node [auto] {$\scriptstyle{}$} (A1_6);
\path (A1_6) edge [-] node [auto] {$\scriptstyle{}$} (A1_7);
\path (A1_7) edge [-] node [auto] {$\scriptstyle{}$} (A1_8);
  \end{tikzpicture}$

\item $\begin{tikzpicture}[xscale=1.0,yscale=1,baseline={(0,0)}]
    \node at (1, .4) {$2$};
    \node at (2, .4) {$3$};
    \node at (3, .4) {$\red{\circled{\emph{3}}}$};
    \node at (4, .4) {$\red{\circled{\emph{2}}}$};
        \node at (5, .4) {$\red{\circled{\emph{2}}}$};
\node at (6, .4) {$\red{\circled{\emph{3}}}$};
\node at (7, .4) {$3$};
\node at (8, .4) {$2$};
                                
    \node (A1_1) at (1, 0) {$\bullet$};
    \node (A1_2) at (2, 0) {$\bullet$};
    \node (A1_3) at (3, 0) {$\bullet$};
    \node (A1_4) at (4, 0) {$\bullet$};
        \node (A1_5) at (5, 0) {$\bullet$};
            \node (A1_6) at (6, 0) {$\bullet$};
\node (A1_7) at (7, 0) {$\bullet$};
\node (A1_8) at (8, 0) {$\bullet$};

    \path (A1_2) edge [-] node [auto] {$\scriptstyle{}$} (A1_3);
    \path (A1_3) edge [-] node [auto] {$\scriptstyle{}$} (A1_4);
    \path (A1_1) edge [-] node [auto] {$\scriptstyle{}$} (A1_2);
        \path (A1_4) edge [-] node [auto] {$\scriptstyle{}$} (A1_5);
\path (A1_5) edge [-] node [auto] {$\scriptstyle{}$} (A1_6);
\path (A1_6) edge [-] node [auto] {$\scriptstyle{}$} (A1_7);
\path (A1_7) edge [-] node [auto] {$\scriptstyle{}$} (A1_8);
  \end{tikzpicture}$
\end{enumerate}
\end{lem}
\begin{proof}
For each of these graphs the procedure is the same. Using GAP, we list all the matrices corresponding to embeddings of the graph as described above. One can then verify that for each of these matrices the rows representing to the red/circled vertices corresponds to a standard embedding. Explicitly, this means that by permuting columns and multiplying columns by $-1$ the four rows representing the red/circled vertices can be put in the form
\[\begin{pmatrix}
1&1&1&0&0&0&0&0&\dots&0\\
0&0&-1&1&0&0&0&0&\dots&0\\
0&0&0&-1&1&0&0&0&\dots&0\\
0&0&0&0&-1&1&1&0&\dots&0
\end{pmatrix}
\]
We now list the relevant matrices for each of the graphs in turn. The GAP code generating these matrices is given in Appendix~\ref{sec:code}. In all cases the rows corresponding to the red/circled vertices are those between the horizontal lines in the matrices and one can easily see that they have the necessary properties.

\setcounter{MaxMatrixCols}{13}

\begin{enumerate}
\item For the plumbing graph $\begin{tikzpicture}[xscale=1.0,yscale=1,baseline={(0,0)}]
    \node at (1, .4) {$2$};
    \node at (2, .4) {$\red{\circled{3}}$};
    \node at (3, .4) {$\red{\circled{2}}$};
    \node at (4, .4) {$\red{\circled{2}}$};
        \node at (5, .4) {$\red{\circled{3}}$};
                \node at (6, .4) {$2$};
    \node (A1_1) at (1, 0) {$\bullet$};
    \node (A1_2) at (2, 0) {$\bullet$};
    \node (A1_3) at (3, 0) {$\bullet$};
    \node (A1_4) at (4, 0) {$\bullet$};
        \node (A1_5) at (5, 0) {$\bullet$};
            \node (A1_6) at (6, 0) {$\bullet$};
    \path (A1_2) edge [-] node [auto] {$\scriptstyle{}$} (A1_3);
    \path (A1_3) edge [-] node [auto] {$\scriptstyle{}$} (A1_4);
    \path (A1_1) edge [-] node [auto] {$\scriptstyle{}$} (A1_2);
        \path (A1_4) edge [-] node [auto] {$\scriptstyle{}$} (A1_5);
            \path (A1_5) edge [-] node [auto] {$\scriptstyle{}$} (A1_6);
  \end{tikzpicture}$
there is one possible matrix:

\vspace{.3cm}
\[\begin{pmatrix}
0&0&0&-1&1&0&0&0&0\\
\hline
-1&0&0&1&0&0&0&1&0\\
1&-1&0&0&0&0&0&0&0\\
0&1&-1&0&0&0&0&0&0\\
0&0&1&0&0&-1&0&0&1\\
\hline
0&0&0&0&0&1&1&0&0
\end{pmatrix}
\]

\item For the plumbing graph
$\begin{tikzpicture}[xscale=1.0,yscale=1,baseline={(0,0)}]
    \node at (1, .4) {$2$};
    \node at (2, .4) {$\red{\circled{3}}$};
    \node at (3, .4) {$\red{\circled{2}}$};
    \node at (4, .4) {$\red{\circled{2}}$};
        \node at (5, .4) {$\red{\circled{3}}$};
\node at (6, .4) {$3$};
\node at (7, .4) {$2$};
                                
    \node (A1_1) at (1, 0) {$\bullet$};
    \node (A1_2) at (2, 0) {$\bullet$};
    \node (A1_3) at (3, 0) {$\bullet$};
    \node (A1_4) at (4, 0) {$\bullet$};
        \node (A1_5) at (5, 0) {$\bullet$};
            \node (A1_6) at (6, 0) {$\bullet$};
\node (A1_7) at (7, 0) {$\bullet$};

    \path (A1_2) edge [-] node [auto] {$\scriptstyle{}$} (A1_3);
    \path (A1_3) edge [-] node [auto] {$\scriptstyle{}$} (A1_4);
    \path (A1_1) edge [-] node [auto] {$\scriptstyle{}$} (A1_2);
        \path (A1_4) edge [-] node [auto] {$\scriptstyle{}$} (A1_5);
\path (A1_5) edge [-] node [auto] {$\scriptstyle{}$} (A1_6);
\path (A1_6) edge [-] node [auto] {$\scriptstyle{}$} (A1_7);
  \end{tikzpicture}$ there are two matrices:
  \vspace{.3cm}
\[\begin{pmatrix}
0&0&0&0&0&-1&1&0&0&0\\
\hline
0&0&-1&0&0&1&0&1&0&0\\
0&0&1&-1&0&0&0&0&0&0\\
0&0&0&1&-1&0&0&0&0&0\\
1&-1&0&0&1&0&0&0&0&0\\
\hline
-1&0&0&0&0&0&0&0&1&1\\
1&1&0&0&0&0&0&0&0&0
\end{pmatrix}\]

\vspace{.3cm}
\[
\begin{pmatrix}
0&0&0&-1&1&0&0&0&0&0&0 \\
\hline
-1&0 &0&1 &0&0&0&0&1&0&0\\
1 &-1&0&0&0&0&0&0&0&0&0\\
0 &1&-1&0&0&0&0&0&0&0&0\\
0 &0&1&0&0&0&0&-1&0&1&0\\
\hline
0 &0&0&0&0&-1&0&1&0&0&1\\
 0&0&0&0&0&1&1&0&0&0&0
\end{pmatrix}\]

\item For the plumbing graph $\begin{tikzpicture}[xscale=1.0,yscale=1,baseline={(0,0)}]
    \node at (1, .4) {$2$};
    \node at (2, .4) {$2$};
    \node at (3, .4) {$\red{\circled{3}}$};
    \node at (4, .4) {$\red{\circled{2}}$};
        \node at (5, .4) {$\red{\circled{2}}$};
\node at (6, .4) {$\red{\circled{3}}$};
\node at (7, .4) {$3$};
                                
    \node (A1_1) at (1, 0) {$\bullet$};
    \node (A1_2) at (2, 0) {$\bullet$};
    \node (A1_3) at (3, 0) {$\bullet$};
    \node (A1_4) at (4, 0) {$\bullet$};
        \node (A1_5) at (5, 0) {$\bullet$};
            \node (A1_6) at (6, 0) {$\bullet$};
\node (A1_7) at (7, 0) {$\bullet$};

    \path (A1_2) edge [-] node [auto] {$\scriptstyle{}$} (A1_3);
    \path (A1_3) edge [-] node [auto] {$\scriptstyle{}$} (A1_4);
    \path (A1_1) edge [-] node [auto] {$\scriptstyle{}$} (A1_2);
        \path (A1_4) edge [-] node [auto] {$\scriptstyle{}$} (A1_5);
\path (A1_5) edge [-] node [auto] {$\scriptstyle{}$} (A1_6);
\path (A1_6) edge [-] node [auto] {$\scriptstyle{}$} (A1_7);
  \end{tikzpicture}$
there are two matrices:
\vspace{.3cm}
\[\begin{pmatrix}
1&-1&0&0&0&0&0&0&0&0\\
-1&0&1&0&0&0&0&0&0&0\\
\hline
1&1&0&-1&0&0&0&0&0&0\\
0&0&0&1&-1&0&0&0&0&0\\
0&0&0&0&1&-1&0&0&0&0\\
0&0&0&0&0&1&-1&1&0&0\\
\hline
0&0&0&0&0&0&1&0&1&1
\end{pmatrix}\]

\vspace{.3cm}
\[
\begin{pmatrix}
0&0&0&-1&1&0&0&0&0&0&0\\
0&0&0&1&0&-1&0&0&0&0&0\\
\hline
-1&0&0&0&0&1&1&0&0&0&0\\
1&-1&0&0&0&0&0&0&0&0&0\\
0&1&-1&0&0&0&0&0&0&0&0\\
0&0&1&0&0&0&0&-1&1&0&0\\
\hline
0&0&0&0&0&0&0&1&0&1&1
\end{pmatrix}
\]

\item For the plumbing graph $$\begin{tikzpicture}[xscale=1.0,yscale=1,baseline={(0,0)}]
    \node at (1, .4) {$2$};
    \node at (2, .4) {$2$};
    \node at (3, .4) {$\red{\circled{3}}$};
    \node at (4, .4) {$\red{\circled{2}}$};
        \node at (5, .4) {$\red{\circled{2}}$};
\node at (6, .4) {$\red{\circled{3}}$};
\node at (7, .4) {$4$};
\node at (8, .4) {$2$};
                                
    \node (A1_1) at (1, 0) {$\bullet$};
    \node (A1_2) at (2, 0) {$\bullet$};
    \node (A1_3) at (3, 0) {$\bullet$};
    \node (A1_4) at (4, 0) {$\bullet$};
        \node (A1_5) at (5, 0) {$\bullet$};
            \node (A1_6) at (6, 0) {$\bullet$};
\node (A1_7) at (7, 0) {$\bullet$};
\node (A1_8) at (8, 0) {$\bullet$};

    \path (A1_2) edge [-] node [auto] {$\scriptstyle{}$} (A1_3);
    \path (A1_3) edge [-] node [auto] {$\scriptstyle{}$} (A1_4);
    \path (A1_1) edge [-] node [auto] {$\scriptstyle{}$} (A1_2);
        \path (A1_4) edge [-] node [auto] {$\scriptstyle{}$} (A1_5);
\path (A1_5) edge [-] node [auto] {$\scriptstyle{}$} (A1_6);
\path (A1_6) edge [-] node [auto] {$\scriptstyle{}$} (A1_7);
\path (A1_7) edge [-] node [auto] {$\scriptstyle{}$} (A1_8);
  \end{tikzpicture}$$ there are four matrices:
\vspace{.3cm}
\[\begin{pmatrix}
1&-1&0&0&0&0&0&0&0&0&0\\
-1&0&1&0&0&0&0&0&0&0&0\\
\hline
1&1&0&0&0&-1&0&0&0&0&0\\
0&0&0&0&0&1&-1&0&0&0&0\\
0&0&0&0&0&0&1&1&0&0&0\\
0&0&0&1&-1&0&0&-1&0&0&0\\
\hline
0&0&0&-1&0&0&0&0&1&1&1\\
0&0&0&1&1&0&0&0&0&0&0
\end{pmatrix}\]

\vspace{.3cm}
\[\begin{pmatrix}
1&-1&0&0&0&0&0&0&0&0&0&0\\
-1&0&1&0&0&0&0&0&0&0&0&0\\
\hline
1&1&0&-1&0&0&0&0&0&0&0&0\\
0&0&0&1&-1&0&0&0&0&0&0&0\\
0&0&0&0&1&1&0&0&0&0&0&0\\
0&0&0&0&0&-1&0&0&1&1&0&0\\
\hline
0&0&0&0&0&0&-1&0&-1&0&1&1\\
0&0&0&0&0&0&1&1&0&0&0&0
\end{pmatrix}\]

\vspace{.3cm}
\[\begin{pmatrix}
0&0&0&0&0&-1&1&0&0&0&0&0\\
0&0&0&0&0&1&0&-1&0&0&0&0\\
\hline
0&0&-1&0&0&0&0&1&1&0&0&0\\
0&0&1&-1&0&0&0&0&0&0&0&0\\
0&0&0&1&1&0&0&0&0&0&0&0\\
1&-1&0&0&-1&0&0&0&0&0&0&0\\
\hline
-1&0&0&0&0&0&0&0&0&1&1&1\\
1&1&0&0&0&0&0&0&0&0&0&0
\end{pmatrix}\]

\vspace{.3cm}
\[\begin{pmatrix}
0&0&0&-1&1&0&0&0&0&0&0&0&0\\
0&0&0&1&0&-1&0&0&0&0&0&0&0\\
\hline
-1&0&0&0&0&1&1&0&0&0&0&0&0\\
1&-1&0&0&0&0&0&0&0&0&0&0&0\\
0&1&1&0&0&0&0&0&0&0&0&0&0\\
0&0&-1&0&0&0&0&0&0&1&1&0&0\\
\hline
0&0&0&0&0&0&0&-1&0&-1&0&1&1\\
0&0&0&0&0&0&0&1&1&0&0&0&0
\end{pmatrix}\]

\item For the plumbing graph $$\begin{tikzpicture}[xscale=1.0,yscale=1,baseline={(0,0)}]
    \node at (1, .4) {$2$};
    \node at (2, .4) {$3$};
    \node at (3, .4) {$\red{\circled{3}}$};
    \node at (4, .4) {$\red{\circled{2}}$};
        \node at (5, .4) {$\red{\circled{2}}$};
\node at (6, .4) {$\red{\circled{3}}$};
\node at (7, .4) {$3$};
\node at (8, .4) {$2$};
                                
    \node (A1_1) at (1, 0) {$\bullet$};
    \node (A1_2) at (2, 0) {$\bullet$};
    \node (A1_3) at (3, 0) {$\bullet$};
    \node (A1_4) at (4, 0) {$\bullet$};
        \node (A1_5) at (5, 0) {$\bullet$};
            \node (A1_6) at (6, 0) {$\bullet$};
\node (A1_7) at (7, 0) {$\bullet$};
\node (A1_8) at (8, 0) {$\bullet$};

    \path (A1_2) edge [-] node [auto] {$\scriptstyle{}$} (A1_3);
    \path (A1_3) edge [-] node [auto] {$\scriptstyle{}$} (A1_4);
    \path (A1_1) edge [-] node [auto] {$\scriptstyle{}$} (A1_2);
        \path (A1_4) edge [-] node [auto] {$\scriptstyle{}$} (A1_5);
\path (A1_5) edge [-] node [auto] {$\scriptstyle{}$} (A1_6);
\path (A1_6) edge [-] node [auto] {$\scriptstyle{}$} (A1_7);
\path (A1_7) edge [-] node [auto] {$\scriptstyle{}$} (A1_8);
  \end{tikzpicture}$$
there are five matrices:
\vspace{.3cm}
\[\begin{pmatrix}
0&1&-1&0&0&0&0&0&0\\
1&-1&0&-1&0&0&0&0&0\\
\hline
0&1&1&0&0&0&-1&0&0\\
0&0&0&0&0&0&1&-1&0\\
0&0&0&0&0&0&0&1&-1\\
0&0&0&0&1&-1&0&0&1\\
\hline
1&0&0&1&-1&0&0&0&0\\
0&0&0&0&1&1&0&0&0
\end{pmatrix}\]

\vspace{.3cm}
\[\begin{pmatrix}
1&-1&0&0&0&0&0&0&0&0&0\\
-1&0&0&0&0&0&0&1&1&0&0\\
\hline
1&1&0&0&-1&0&0&0&0&0&0\\
0&0&0&0&1&-1&0&0&0&0&0\\
0&0&0&0&0&1&-1&0&0&0&0\\
0&0&1&-1&0&0&1&0&0&0&0\\
\hline
0&0&-1&0&0&0&0&0&0&1&1\\
0&0&1&1&0&0&0&0&0&0&0
\end{pmatrix}\]

\vspace{.3cm}
\[\begin{pmatrix}
1&-1&0&0&0&0&0&0&0&0&0&0\\
-1&0&0&0&0&0&0&0&0&1&1&0\\
\hline
1&1&-1&0&0&0&0&0&0&0&0&0\\
0&0&1&-1&0&0&0&0&0&0&0&0\\
0&0&0&1&-1&0&0&0&0&0&0&0\\
0&0&0&0&1&0&0&-1&1&0&0&0\\
\hline
0&0&0&0&0&-1&0&1&0&0&0&1\\
0&0&0&0&0&1&1&0&0&0&0&0
\end{pmatrix}\]

\vspace{.3cm}
$$\begin{pmatrix}
0&0&0&0&0&-1&1&0&0&0&0&0\\
0&0&0&0&0&1&0&-1&0&1&0&0\\
\hline
0&0&-1&0&0&0&0&1&1&0&0&0\\
0&0&1&-1&0&0&0&0&0&0&0&0\\
0&0&0&1&-1&0&0&0&0&0&0&0\\
1&-1&0&0&1&0&0&0&0&0&0&0\\
\hline
-1&0&0&0&0&0&0&0&0&0&1&1\\
1&1&0&0&0&0&0&0&0&0&0&0
\end{pmatrix}$$

\vspace{.3cm}

\[\begin{pmatrix}
0&0&0&-1&1&0&0&0&0&0&0&0&0\\
0&0&0&1&0&0&0&-1&0&0&0&1&0\\
\hline
-1&0&0&0&0&0&0&1&0&1&0&0&0\\
1&-1&0&0&0&0&0&0&0&0&0&0&0\\
0&1&-1&0&0&0&0&0&0&0&0&0&0\\
0&0&1&0&0&0&0&0&-1&0&1&0&0\\
\hline
0&0&0&0&0&-1&0&0&1&0&0&0&1\\
0&0&0&0&0&1&1&0&0&0&0&0&0
\end{pmatrix}\]

\end{enumerate}
\end{proof}

The following lemma is necessary to justify Lemma~\ref{lem:min_example_properties}.
\begin{lem}\label{lem:computer2}
The following linear plumbing graph is rigid:
$$\begin{tikzpicture}[xscale=1.0,yscale=1,baseline={(0,0)}]
    \node at (-1.0, .4) {$2$};
    \node at (0.0, .4) {$4$};
    \node at (1.0, .4) {$2$};
    \node at (2.0, .4) {$2$};
    \node at (3.0, .4) {$2$};
        \node at (4.0, .4) {$3$};
                \node at (5.0, .4) {$2$};
        \node (A1_9) at (-1, 0) {$\bullet$};
    \node (A1_0) at (0, 0) {$\bullet$};
    \node (A1_1) at (1, 0) {$\bullet$};
    \node (A1_2) at (2, 0) {$\bullet$};
    \node (A1_3) at (3, 0) {$\bullet$};
        \node (A1_4) at (4, 0) {$\bullet$};
                \node (A1_5) at (5, 0) {$\bullet$};
        \path (A1_9) edge [-] node [auto] {$\scriptstyle{}$} (A1_0);
    \path (A1_0) edge [-] node [auto] {$\scriptstyle{}$} (A1_1);
    \path (A1_1) edge [-] node [auto] {$\scriptstyle{}$} (A1_2);
        \path (A1_2) edge [-] node [auto] {$\scriptstyle{}$} (A1_3);
                \path (A1_3) edge [-] node [auto] {$\scriptstyle{}$} (A1_4);
                                \path (A1_4) edge [-] node [auto] {$\scriptstyle{}$} (A1_5);
  \end{tikzpicture}.$$
\end{lem}
\begin{proof}
Using GAP we see that this graph admits a unique embedding. The GAP code verifying this is given in Appendix~\ref{sec:code}.
\end{proof}

\section{Code}\label{sec:code}
The following GAP code will carry out the calculations in Lemma~\ref{lem:computer1} and Lemma~\ref{lem:computer2}.
\begin{alltt}
ListEmbeddings:= function(coefs)
local embeddings, k, M;

M:=DiagonalMat(coefs);
for k in [1..Length(coefs)-1] do
M[k][k+1]:=-1;
M[k+1][k]:=-1;
od;

embed:=OrthogonalEmbeddings(M);

Print("----", coefs ,"----","\textbackslash{n}");
Print("Number of embeddings:",Length(embed.solutions),"\textbackslash{n}");
for k in [1..Length(embed.solutions)] do
Print(TransposedMat(embed.vectors\{embed.solutions[k]\}),"\textbackslash{n}");
od;
Print("--------","\textbackslash{n}");
end;

Print("Embeddings for Lemma A.1:\textbackslash{n}");
ListEmbeddings([2,3,2,2,3,2]);
ListEmbeddings([2,3,2,2,3,3,2]);
ListEmbeddings([2,2,3,2,2,3,3]);
ListEmbeddings([2,2,3,2,2,3,4,2]);
ListEmbeddings([2,3,3,2,2,3,3,2]);

Print("Embeddings for Lemma A.2:\textbackslash{n}");
ListEmbeddings([2,4,2,2,2,3,2]);

\end{alltt}

\end{document}